\documentclass{amsart}

\usepackage[T1]{fontenc}
\usepackage{enumerate, amsmath, amsfonts, amssymb, amsthm, mathrsfs, wasysym, graphics, graphicx, xcolor, url, hyperref, hypcap, xargs, multicol, pdflscape, multirow, hvfloat, array, ae, aecompl, pifont, mathtools, a4wide, float, blkarray}
\usepackage{marginnote}
\hypersetup{colorlinks=true, citecolor=darkblue, linkcolor=darkblue}
\usepackage[all]{xy}
\usepackage{quiver}
\usepackage{caption}
\usepackage[export]{adjustbox}

\usepackage[capitalise,noabbrev]{cleveref}

\newtheorem{theorem}{Theorem}[section]
\newtheorem{corollary}[theorem]{Corollary}
\newtheorem{proposition}[theorem]{Proposition}
\newtheorem{lemma}[theorem]{Lemma}

\theoremstyle{definition}
\newtheorem{definition}[theorem]{Definition}
\newtheorem{example}[theorem]{Example}
\newtheorem{remark}[theorem]{Remark}
\newtheorem{question}[theorem]{Question}
\newtheorem{notation}[theorem]{Notation}
\newtheorem{terminology}[theorem]{Terminology}

\makeatletter
\def\l@section{\@tocline{1}{2pt}{0pc}{}{}}
\makeatother
\let\oldtocpart=\tocpart
\renewcommand{\tocpart}[2]{\bf\large\oldtocpart{#1}{#2}}
\let\oldtocsection=\tocsection
\renewcommand{\tocsection}[2]{\bf\oldtocsection{#1}{#2}}

\newcommand{\infl}{\rightarrowtail}
\newcommand{\defl}{\twoheadrightarrow}

\newcommand{\rb}{\mathbb{R}}

\newcommand{\asso}{\mathbb{A}}
\newcommand{\field}{\mathbb{K}}

\definecolor{darkblue}{rgb}{0,0,0.7}
\newcommand{\darkblue}{\color{darkblue}}
\newcommand{\defn}[1]{\textsl{\darkblue #1}}
\hypersetup{colorlinks=true, citecolor=darkblue, linkcolor=darkblue}

\newcommand{\bsm}{\begin{smallmatrix}}
\newcommand{\esm}{\end{smallmatrix}}

\title[Applications of extriangulated categories]{Some applications of extriangulated categories}

\author{Yann Palu}
\address{LAMFA, Universit\'e Picardie Jules Verne, 80039 Amiens, France}
\email{yann.palu@u-picardie.fr}

\thanks{Most diagrams were made using the amazing \href{q.uiver.app}{https://q.uiver.app/}.
This work was partially supported by the French ANR grants SC3A~(15-CE40-0004-01) and CHARMS~(19-CE40-0017)}
\begin{document}

\begin{abstract}
Extriangulated categories axiomatize extension-closed subcategories of triangulated categories and generalise both exact categories and triangulated categories. This survey article presents three applications of extriangulated categories to homotopical algebra, algebraic combinatorics and representation theory.
The first shows that, via some generalised Hovey's correspondence, extriangulated categories easily give rise to model category structures with triangulated homotopy categories. 
As a second application, extriangulated structures play a fondamental role in  the construction of polytopal realisations of $g$-vector fans. This allows for a generalisation of ABHY's construction appearing in the study of scattering amplitudes in theoretical physics.
Lastly, extriangulated categories provide a convenient framework for studying mutations in representation theory and flips in algebraic combinatorics. In nice enough hereditary extriangulated categories, there is a well-behaved theory of mutation for silting objects, which encompass cluster tilting, two-term silting, relative tilting, mutation of maximal almost-rigid modules, flip of dissections and mutation of intermediate co-$t$-structures.
\end{abstract}

\maketitle

\tableofcontents

\section*{Introduction}

In a similar way that Quillen exact categories axiomatise extension-closed full subcategories of abelian categories, extriangulated categories, introduced in~\cite{NakaokaPalu}, axiomatise extension-closed full subcategories of triangulated categories.
They also serve as a tool for unifying proofs in both exact categories and triangulated categories.
One strength of extriangulated categories is that many constructions that are often performed on exact or triangulated categories may leave the realm of exact or triangulated categories: This is the case, for example, of quotients by projective-injective objects in non-Frobenius exact categories, of taking extension-closed subcategories or relative structures of triangulated categories.
However, in all three examples, the resulting categories remain extriangulated.

In this survey article, we illustrate the use of extriangulated structures with three applications.

\vspace*{7pt}
\noindent \emph{Hovey's correspondence} \cite{Hovey-CPmodel,Gillespie-exact}
is a tool for constructing model category structures (see~\cref{def:model}) out of very few information, namely that of fibrant and of cofibrant objects, or alternatively, that of (co)fibrant and trivially (co)fibrant objects.
In~\cite{NakaokaPalu}, we generalised Hovey's correspondence from the case of abelian and exact categories to that of extriangulated categories (see~\cref{thm:extriangulatedHovey}).
Moreover, we proved that the homotopy categories of thus constructed model structures are always triangulated (see~\cref{theorem: homotopy triangulated}).
This gives many examples of localisations of triangulated categories that are both triangulated and equivalent to some ideal subquotient.

\vspace*{7pt}
\noindent \emph{Polytopal realisations of $g$-vector fans}.
Inspired by \cite{BazierMatteChapelierLagetDouvilleMousavandThomasYildirim}, we revisit in~\cite{PadrolPaluPilaudPlamondon}, Arkani-Hamed--Bai--He--Yan's realisation~\cite{Arkani-HamedBaiHeYan} of the associahedron, that was motivated by the theory of scattering amplitudes in theoretical physics.
We do so by relating ABHY's construction to the so-called type cone of Peter McMullen (\cref{def:typeCone}).
This allows us to generalise ABHY's construction and to give all polytopal realisations of $g$-vector fans for cluster algebras of finite type (\cref{theorem:PPPPpartII}).
Our approach only relies on two purely combinatorial properties of the type cone (\cref{theorem:PPPPpartI}) and we only use representation theory, and extriangulated structures, in order to prove that those two properties are satisfied in the context of cluster algebras.
Working with extriangulated categories permits to apply this strategy to other $g$-vector fans: those of (brick and 2-acyclic) gentle algebras.

\vspace*{7pt}
\noindent \emph{Mutation}.
Initially motivated by the additive categorification of cluster algebras, many instances of mutation have been studied in representation theory.
Classes of objects that enjoy a nice theory of mutation include cluster-tilting objects, two-term silting objects, relative tilting objects, maximal almost rigid objects.
It turns out that, in all those examples, some nice hereditary extriangulated structures arise.
In~\cite{GorskyNakaokaPalu-Mutation} (see \cref{section:mutation}), we introduced the notion of a 0-Auslander extriangulated category and proved there is a well-behaved theory of mutation for their maximal rigid objects.
Moreover, the following are examples of 0-Auslander extriangulated categories:
\begin{itemize}
 \item Cluster categories (more generally, 2-Calabi--Yau triangulated categories) with some cluster tilting object $T$ and endowed with the maximal extriangulated substructure making $T$ projective. The maximal rigid objects are the cluster tilting objects.
 \item The homotopy category of two-term complexes of projectives over a finite-dimensional algebra. Its maximal rigid objects are the two-term silting objects.
 \item The category of representations of a quiver of Dynkin type $A_n$ over some field, endowed with the relative exact structure given by the Auslander--Reiten sequences with non-indecomposable middle terms. Its maximal rigid objects are the maximal almost-rirgid objects.
 \item More examples are given in \cref{ex:0-Auslander,ex:max rigid in 0-Auslander}.
\end{itemize}
We thus recover many notions of mutation, and we hope that 0-Auslander extriangulated structures will be useful when studying new mutation theories in the future.

\section*{Acknowledgements}
I would like to warmly thank the organizers of the Abel Symposium 2022, and more specifically Petter Andreas Bergh for his encouragements in writing this survey.
The applications described in \cref{section:Hovey,section:g-vectors,section:mutation} were the result of teamwork; I thank my collaborators for the fun we had in inventing and proving those results together.
Specific thanks to the other three P's for useful comments on a previous version of these notes.
My deepest and friendliest thanks go to Hiroyuki Nakaoka, without whom this whole story wouldn't have started.

\section{The definition of an extriangulated category}

\subsection{Prototypical example}
\label{subsection: motivation extricats}

We are interested in the similarities between the following two results:

\begin{theorem}[Happel \cite{Happel-book}]
\label{theorem: Happel}
The stable category of a Frobenius exact category is triangulated.
\end{theorem}

\begin{theorem}[Iyama--Yoshino \cite{IyamaYoshino}]
\label{theorem: IyamaYoshino}
Let $\mathcal{R}$ be a functorially finite, rigid subcategory in a 2-Calabi--Yau triangulated category.
Then, the subquotient $\mathcal{R}^{\perp_1}/[\mathcal{R}]$ is triangulated.
\end{theorem}

Here, $\mathcal{R}^{\perp_1}$ is the full subcategory whose objects $X$ satisfy $\operatorname{Ext}^1(R,X)=0$ or, equivalently by the 2-Calabi--Yau property, $\operatorname{Ext}^1(X,R)=0$. 
The proof of \cref{theorem: IyamaYoshino} resembles that of \cref{theorem: Happel} where exact categories (that are morally full extension-closed subcategories of abelian categories) are replaced by $\mathcal{R}^{\perp_1}$, which is a full extension-closed subcategory of a triangulated category; where conflations are replaced by triangles; and where the use of push-outs or pull-backs is replaced by that of the octahedral axiom.

One of the motivations for the introduction of extriangulated categories was to give a framework that would formalise the analogy above.
Hence, the class of extriangulated categories has to contain both exact categories and extension-closed full subcategories of triangulated categories.
Extriangulated categories should thus come endowed with some class of diagrams that generalise both conflations $X\infl Y\defl Z$ of exact categories and triangles $X\to Y\to Z\xrightarrow{\delta}\Sigma X$.
However, for $X,Y,Z$ in some extension-closed subcategory of a triangulated category, the shift $\Sigma X$ generally does not belong to the subcategory.
One might be tempted to only consider the part $X\to Y\to Z$ of the triangle above, but this would be a mistake: this complex does not retain enough information to know which commutative diagrams are morphisms of triangles.
We should thus keep track of the morphism $\delta$, while not remembering the object $\Sigma X$.
This leads us to considering diagrams of the form
$X\infl Y\defl Z\overset{\delta}{\dashrightarrow}$, with
\begin{itemize}
 \item $\delta\in\mathbb{E}(Z,X)=\operatorname{Ext}^1(Z,X)$ where $\delta$ classifies the conflation, if the category is exact, and
 \item  $\delta\in\mathbb{E}(Z,X)=\operatorname{Hom}(Z,\Sigma X)$ is the connecting morphism in the triangle, if the category is extension-closed in a triangulated category.
\end{itemize}
With these considerations in mind, we can now introduce the formal definitions.

\subsection{Definitions}

Fix an additive category $\mathscr{C}$, and an additive bifunctor $\mathbb{E}: \mathscr{C}^\mathrm{op}\times\mathscr{C}\to\mathit{Ab}$, where $\mathit{Ab}$ is the category of abelian groups.

\begin{definition}
For any~$X,Z\in\mathscr{C}$, an element~$\delta\in\mathbb{E}(Z,X)$ is called an~\defn{$\mathbb{E}$-extension}.
A \defn{split}~$\mathbb{E}$-extension is a zero element~$0\in\mathbb{E}(Z,X)$, for some objects $X,Z\in\mathscr{C}$.
By additivity of $\mathscr{C}$ and $\mathbb{E}$ we can define, for any two $\mathbb{E}$-extensions $\delta\in\mathbb{E}(Z,X)$ and $\delta'\in\mathbb{E}(Z',X')$, the $\mathbb{E}$-extension
\[
\delta\oplus\delta'\in\mathbb{E}(Z\oplus Z',X\oplus X'),
\]
as being the image of $(\delta,0,0,\delta')$ under the natural isomorphism
$\mathbb{E}(Z,X)\oplus\mathbb{E}(Z',X)\oplus\mathbb{E}(Z,X')\oplus\mathbb{E}(Z',X')\cong\mathbb{E}(Z\oplus Z',X\oplus X')$.
\end{definition}

\begin{remark}
Let $\delta\in\mathbb{E}(Z,X)$ be an $\mathbb{E}$-extension. By functoriality, any morphisms $f\in\mathscr{C}(X,X')$ and $h\in\mathscr{C}(Z',Z)$ induce $\mathbb{E}$-extensions $f_\ast\delta=\mathbb{E}(Z,f)(\delta)$ in $\mathbb{E}(Z,X')$ and $h^\ast\delta=\mathbb{E}(h,X)(\delta)$ in $\mathbb{E}(Z',X)$.
Using those notations, we have $\mathbb{E}(h,f)(\delta)=h^\ast f_\ast\delta=f_\ast h^\ast\delta$ in $\mathbb{E}(Z',X')$.
\end{remark}

\begin{definition}
A \defn{morphism} $(f,h):\delta\to\delta'$ of $\mathbb{E}$-extensions from $\delta\in\mathbb{E}(Z,X)$ to $\delta'\in\mathbb{E}(Z',X')$ is a pair of morphisms $f\in\mathscr{C}(X,X')$ and $h\in\mathscr{C}(Z,Z')$ in $\mathscr{C}$, such that $f_\ast\delta=h^\ast\delta'$.
\end{definition}

\begin{example}
\label{ex:morphisms of extensions}
 If $\mathscr{C}$ is a triangulated category, and $\mathbb{E}=\mathscr{C}(-,\Sigma-)$, then $h^\ast\delta'=\delta'\circ h$, $f_\ast\delta = (\Sigma f)\circ\delta$ and a morphism of extensions is a commutative square:
\[\begin{tikzcd}
	Z & {\Sigma X} \\
	{Z'} & {\Sigma X'}
	\arrow["\delta", from=1-1, to=1-2]
	\arrow["h"', from=1-1, to=2-1]
	\arrow["\Sigma f", from=1-2, to=2-2]
	\arrow["{\delta'}", from=2-1, to=2-2]
\end{tikzcd}\]
which induces a morphism of triangles.
If $\mathscr{C}$ is an exact category then the condition $f_\ast\delta=h^\ast\delta'$ equivalently says that $(f,h)$ can be completed to a morphism $(f,g,h)$ of conflations~\cite[Proposition 3.1]{Buhler-Exact}.
\end{example}

\begin{definition}
\label{DefSqEquiv}
Let $X,Z\in\mathscr{C}$ be any two objects. Two sequences of morphisms in $\mathscr{C}$
\[
X \overset{x}{\longrightarrow}Y\overset{y}{\longrightarrow}Z \text{ and } X\overset{x'}{\longrightarrow}Y'\overset{y'}{\longrightarrow}Z
\]
are said to be \defn{equivalent} if there exists an isomorphism $g\in\mathscr{C}(Y,Y')$ such that the following diagram commutes.
\[
\xy
(-20,0)*+{X}="X";
(5,0)*+{}="1";
(0,6)*+{Y}="Y";
(0,-6)*+{Y'}="Y'";
(-5,0)*+{}="2";
(20,0)*+{Z}="Z";
{\ar^{x} "X";"Y"};
{\ar^{y} "Y";"Z"};
{\ar_{x^{\prime}} "X";"Y'"};
{\ar_{y^{\prime}} "Y'";"Z"};
{\ar^{g}_{\cong} "Y";"Y'"};
{\ar@{}|{} "X";"1"};
{\ar@{}|{} "2";"Z"};
\endxy
\]
The equivalence class of $X \xrightarrow{x} Y \xrightarrow{y} Z$ is denoted by $[X \xrightarrow{x} Y \xrightarrow{y} Z]$.
\end{definition}

\begin{notation}
For any $X,Y,Z,A,B,C\in\mathscr{C}$, and any $[X \xrightarrow{x} Y \xrightarrow{y} Z]$, $[A \xrightarrow{a} B \xrightarrow{b} C]$, we let
\[ 0=[X\overset{\left[\bsm1\\0\esm\right]}{\longrightarrow} X\oplus Y \overset{\left[\bsm0\;1\esm\right]}{\longrightarrow} Y]
\]
and
\[
[X \xrightarrow{x} Y \xrightarrow{y} Z]\oplus [A \xrightarrow{a} B \xrightarrow{b} C] = [X\oplus A \overset{\left[\bsm x \; 0 \\ 0\; a\esm\right]}{\longrightarrow} Y\oplus B \overset{\left[\bsm y\;0\\0\;b \esm\right]}{\longrightarrow} Z\oplus C].
\]
\end{notation}

\begin{definition}
A \defn{realization} $\mathfrak{s}$ is a correspondence associating, with any $\mathbb{E}$-extension ${\delta\in\mathbb{E}(Z,X)}$, an equivalence class $\mathfrak{s}(\delta)=[X\xrightarrow{x}Y\xrightarrow{y}Z]$ and satisfying the following condition: 
\begin{itemize}
\item[$(\ast)$] Let $\delta\in\mathbb{E}(Z,X)$ and $\delta'\in\mathbb{E}(Z',X')$ be any pair of $\mathbb{E}$-extensions, with
\[
\mathfrak{s}(\delta)=[X\xrightarrow{x}Y\xrightarrow{y}Z] \text{ and } \mathfrak{s}(\delta')=[X'\xrightarrow{x'}Y'\xrightarrow{y'}Z'].
\]
Then, for any morphism $(f,h):\delta\to\delta'$, there exists $g\in\mathscr{C}(Y,Y')$ such that the following diagram commutes:
\[
\xy
(-12,6)*+{X}="0";
(0,6)*+{Y}="2";
(12,6)*+{Z}="4";
(-12,-6)*+{X'}="10";
(0,-6)*+{Y'}="12";
(12,-6)*+{Z'}="14";
{\ar^{x} "0";"2"};
{\ar^{y} "2";"4"};
{\ar_{f} "0";"10"};
{\ar^{g} "2";"12"};
{\ar^{h} "4";"14"};
{\ar_{x'} "10";"12"};
{\ar_{y'} "12";"14"};
{\ar@{}|{} "0";"12"};
{\ar@{}|{} "2";"14"};
{\ar@{}|\circlearrowright "0";"12"};
{\ar@{}|\circlearrowright "2";"14"};
\endxy
\]
\end{itemize}

The sequence $X\xrightarrow{x}Y\xrightarrow{y}Z$ is said to \defn{realize} $\delta$ if $\mathfrak{s}(\delta)=[X\xrightarrow{x}Y\xrightarrow{y}Z]$, and the triple $(f,g,h)$ is said to realize $(f,h)$ if the previous diagram commutes.
\end{definition}

\begin{definition}
A realization of $\mathbb{E}$ is called an \defn{additive realization} if moreover:
\begin{enumerate}
\item For any $X,Z\in\mathscr{C}$, the realization of the split $\mathbb{E}$-extension $0\in\mathbb{E}(Z,X)$ is given by $\mathfrak{s}(0)=0$.
\item For any two $\mathbb{E}$-extensions $\delta\in\mathbb{E}(Z,X)$ and $\delta'\in\mathbb{E}(Z',X')$, the realization of $\delta\oplus\delta'$ is given by $\mathfrak{s}(\delta\oplus\delta')=\mathfrak{s}(\delta)\oplus\mathfrak{s}(\delta')$.
\end{enumerate}
\end{definition}

\begin{definition}(\cite[Def. 2.12]{NakaokaPalu})
A triple $(\mathscr{C},\mathbb{E},\mathfrak{s})$ is called an \defn{extriangulated category} if the following holds:
\begin{itemize}
\item[{\rm (ET1)}] $\mathbb{E}:\mathscr{C}^{\mathrm{op}}\times\mathscr{C}\to\mathit{Ab}$ is an additive bifunctor;
\item[{\rm (ET2)}] $\mathfrak{s}$ is an additive realization of $\mathbb{E}$;
\item[{\rm (ET3)}] Let $\delta\in\mathbb{E}(Z,X)$ and $\delta'\in\mathbb{E}(Z',X')$ be $\mathbb{E}$-extensions, respectively realized by $X\xrightarrow{x}Y\xrightarrow{y}Z$ and $X'\xrightarrow{x'}Y'\xrightarrow{y'}Z'$. Then, for any commutative square
\[
\xy
(-12,6)*+{X}="0";
(0,6)*+{Y}="2";
(12,6)*+{Z}="4";
(-12,-6)*+{X'}="10";
(0,-6)*+{Y'}="12";
(12,-6)*+{Z'}="14";
{\ar^{x} "0";"2"};
{\ar^{y} "2";"4"};
{\ar_{f} "0";"10"};
{\ar^{g} "2";"12"};
{\ar_{x'} "10";"12"};
{\ar_{y'} "12";"14"};
{\ar@{}|{} "0";"12"};
{\ar@{}|\circlearrowright "0";"12"};
\endxy 
\]
in $\mathscr{C}$, there exists a morphism $(f,h):\delta\to\delta'$ satisfying $h\circ y=y'\circ g$.
\item[{\rm (ET3)$^{\mathrm{op}}$}] Dual of {\rm (ET3)}.
\item[{\rm (ET4)}] Let $\delta\in\mathbb{E}(Z',X)$ and $\delta'\in\mathbb{E}(X',Y)$ be $\mathbb{E}$-extensions realized respectively by
\[
X\overset{f}{\longrightarrow}Y\overset{f'}{\longrightarrow}Z' \quad\text{and}\quad Y\overset{g}{\longrightarrow}Z\overset{g'}{\longrightarrow}X'.
\]
Then there exist an object $Y'\in\mathscr{C}$, a commutative diagram in $\mathscr{C}$
\[\begin{tikzcd}
	X & Y & {Z'} & {} \\
	X & Z & \textcolor{rgb,255:red,51;green,61;blue,255}{Y'} & {} \\
	& {X'} & {X'} \\
	& {} & {}
	\arrow["f", from=1-1, to=1-2]
	\arrow["{f'}", from=1-2, to=1-3]
	\arrow[Rightarrow, no head, from=1-1, to=2-1]
	\arrow["g"', from=1-2, to=2-2]
	\arrow["d", color={rgb,255:red,51;green,61;blue,255}, from=1-3, to=2-3]
	\arrow["\delta", dashed, from=1-3, to=1-4]
	\arrow["h", from=2-1, to=2-2]
	\arrow["{h'}", color={rgb,255:red,51;green,61;blue,255}, from=2-2, to=2-3]
	\arrow["{\delta''}", color={rgb,255:red,51;green,61;blue,255}, dashed, from=2-3, to=2-4]
	\arrow["{g'}"', from=2-2, to=3-2]
	\arrow["{\delta'}"', dashed, from=3-2, to=4-2]
	\arrow["e", color={rgb,255:red,51;green,61;blue,255}, from=2-3, to=3-3]
	\arrow["{f'_\ast\delta'}", dashed, from=3-3, to=4-3]
	\arrow[Rightarrow, no head, from=3-2, to=3-3]
\end{tikzcd}\]

and an $\mathbb{E}$-extension $\delta''\in\mathbb{E}(Y',X)$ realized by $X\overset{h}{\longrightarrow}Z\overset{h'}{\longrightarrow}Y'$, which satisfy the following compatibilities.
  \begin{enumerate}[(i)]
    \item $Z'\overset{d}{\longrightarrow}Y'\overset{e}{\longrightarrow}X'$ realizes $f'_\ast\delta'$,
    \item $d^\ast\delta''=\delta$,
    \item $f_\ast\delta''=e^\ast\delta'$ (equivalently, $(f,e)$ is a morphism from $\delta''$ to $\delta'$ realised by $(f,1_B,e)$). 
  \end{enumerate}
\item[{\rm (ET4)$^{\mathrm{op}}$}] Dual of {\rm (ET4)}.
\end{itemize}
\end{definition}

\begin{remark}
\label{remark:dualExtricat}
Note that the dual of an extriangulated category is again extriangulated.
Thus the proof of a statement also implies the dual statement.
\end{remark}

\begin{terminology}
Let $(\mathscr{C},\mathbb{E},\mathfrak{s})$ be an extriangulated category.
\begin{enumerate}
\item A sequence $X\xrightarrow{i}Y\xrightarrow{p}Z$ is called a \defn{conflation} if it realizes some $\mathbb{E}$-extension in $\mathbb{E}(Z,X)$; in which case the morphism $X\xrightarrow{i}Y$ is called an \defn{inflation}, written $X\infl Y$, and the morphism $Y\xrightarrow{p}Z$ is called a \defn{deflation}, witten $Y\defl Z$.
\item An \defn{extriangle} is a diagram $X\overset{i}{\infl} Y\overset{p}{\defl} Z\overset{\delta}{\dashrightarrow}$ where $X\overset{i}{\infl} Y\overset{p}{\defl} Z$ is a conflation realizing the $\mathbb{E}$-extension $\delta\in\mathbb{E}(Z,X)$.
\item Similarly, we call \defn{morphism of extriangles} (or \defn{morphism of conflations}) any diagram
\[\begin{tikzcd}
	X & Y & Z & {} \\
	A & B & C & {}
	\arrow["i", tail, from=1-1, to=1-2]
	\arrow["p", two heads, from=1-2, to=1-3]
	\arrow["f"', from=1-1, to=2-1]
	\arrow["g"', from=1-2, to=2-2]
	\arrow["h"', from=1-3, to=2-3]
	\arrow["\delta", dashed, from=1-3, to=1-4]
	\arrow["j", tail, from=2-1, to=2-2]
	\arrow["q", two heads, from=2-2, to=2-3]
	\arrow["\zeta", dashed, from=2-3, to=2-4]
\end{tikzcd}\]
where $(f,h):\delta\to\delta'$ is a morphism of $\mathbb{E}$-extensions realized by $(f,g,h)$.
\end{enumerate}
\end{terminology}

\begin{definition}[{\cite[Definition 2.32]{Bennett-TennenhausShah-transport}}]
Let $(\mathscr{C},\mathbb{E},\mathfrak{s})$ and $(\mathscr{D},\mathbb{F},\mathfrak{t})$ be two extriangulated categories.
An \defn{exact functor}, or \defn{extriangulated functor}, from $\mathscr{C}$ to $\mathscr{D}$ is the data of an additive functor $F:\mathscr{C}\to\mathscr{D}$ and of a natural transformation $\Gamma:\mathbb{E}\to\mathbb{F}(F-,F-)$ such that, for any extriangle
$X\overset{i}{\infl}Y\overset{p}{\defl}Z\overset{\delta}{\dashrightarrow}$ in $\mathscr{C}$, the diagram
$FX\overset{Fi}{\infl}FY\overset{Fp}{\defl}FZ\overset{\Gamma_{Z,X}(\delta)}{\dashrightarrow}$
is an extriangle in $\mathscr{D}$.
\end{definition}

There is also a notion of extriangulated natural transformations. Extriangulated categories thus form a 2-category (see~\cite[Definition 4.1 and Corollary 4.15]{Bennett-TennenhausHauglandSandoyShah} for $n$-exangulated categories), which gives a notion of extriangulated equivalence.

\begin{proposition}[{\cite[Proposition 2.13]{NakaokaOgawaSakai} (see also \cite[Proposition 4.11]{Bennett-TennenhausHauglandSandoyShah} and \cite[Proposition 2.14]{HeHeZhou-localization})}]
 An extriangulated funtor $(F,\Gamma):\mathscr{C}\to\mathscr{D}$ is an extriangulated equivalence if and only if $F$ is an equivalence of categories and $\Gamma$ is a natural isomorphism.
\end{proposition}

\begin{definition}[{\cite[Definition 3.7]{Haugland-n-exangulated}}]
Let $(\mathscr{C},\mathbb{E},\mathfrak{s})$ be an extriangulated category.
An \defn{extriangulated subcategory} of $\mathscr{C}$ is the data of an extriangulated category $(\mathscr{C}',\mathbb{E}',\mathfrak{s}')$ and an exact functor $(F,\Gamma):\mathscr{C}'\to\mathscr{C}$ such that $\mathscr{C}'$ is a strictly full subcategory of $\mathscr{C}$, $F$ is the inclusion functor and for all $X,Z\in\mathscr{C}'$, $\Gamma_{Z,X}$ is an inclusion.
\end{definition}

\begin{example}
Any strictly full, extension-closed subcategory of an extriangulated category is canonically an extriangulated subcategory (see \cref{sssection:ext-closed} below).
In particular, given a $t$-structure on a triangulated category $\mathscr{T}$, its heart is an extriangulated subcategory of $\mathscr{T}$. And given a co-$t$-structure on $\mathscr{T}$, its coheart and its extended coheart are extriangulated subcategories.
By \cite[Proposition 3.9 and Theorem 5.1]{Haugland-n-exangulated}, there is a bijection between the subgroups of the Grothendieck group of an extriangulated category and certain of its extriangulated subcategories (that are called complete and dense).
Many more examples of extriangulated subcategories and of extriangulated functors are given in~\cite[Section 5]{Bennett-TennenhausHauglandSandoyShah}.
\end{example}

\subsection{First properties}
Let $(\mathscr{C},\mathbb{E},\mathfrak{s})$ be an extriangulated category.

The axioms above ensure that any extriangle induce six-term exact sequences after application of some covariant or contravariant Hom-functor.
In particular in any conflation, the inflation is a weak kernel of the deflation, and the deflation is a weak cokernel of the inflation.

\begin{proposition}[{\cite[Propositions~3.3~\&~3.11]{NakaokaPalu}}]
\label{prop:extricat long exact sequences}
Assume that $X\xrightarrow{x}Y\xrightarrow{y}Z\overset{\delta}{\dashrightarrow}$ is an extriangle.
Then the following sequences of natural transformations are exact:
\[
\mathscr{C}(Z,-)\overset{-\circ y}{\longrightarrow}\mathscr{C}(Y,-)\overset{-\circ x}{\longrightarrow}\mathscr{C}(X,-)\overset{\delta^\sharp}{\longrightarrow}\mathbb{E}(Z,-)\overset{y^\ast}{\longrightarrow}\mathbb{E}(Y,-)\overset{x^\ast}{\longrightarrow}\mathbb{E}(X,-),
\]
\[
\mathscr{C}(-,X)\overset{x\circ-}{\longrightarrow}\mathscr{C}(-,Y)\overset{y\circ-}{\longrightarrow}\mathscr{C}(-,Z)\overset{\delta_\sharp}{\longrightarrow}\mathbb{E}(-,X)\overset{x_\ast}{\longrightarrow}\mathbb{E}(-,Y)\overset{y_\ast}{\longrightarrow}\mathbb{E}(-,Z),
\]
where~$\delta^\sharp(f)=f_\ast\delta$ and~$\delta_\sharp(g)=g^\ast\delta$.
\end{proposition}

The following, and its dual, are often used as replacements for the non-existent push-outs or pull-backs in extriangulated categories:
\begin{proposition}[{\cite[Lemma A.11]{Chen-PhD}}]
\label{prop:htpPO}
Any two extriangles $A\overset{i}{\infl}B\overset{p}{\defl}C\overset{\delta}{\dashrightarrow}$ and $A\overset{i'}{\infl}B'\overset{p'}{\defl}C'\overset{\delta'}{\dashrightarrow}$ with same source can be completed to a commutative diagram of extriangles
\[\begin{tikzcd}
	A & B & C & {} \\
	{B'} & \textcolor{rgb,255:red,92;green,92;blue,214}{E} & C & {} \\
	{C'} & {C'} & {} & {} \\
	{} & {}
	\arrow["{i'}"', tail, from=1-1, to=2-1]
	\arrow["{p'}"', two heads, from=2-1, to=3-1]
	\arrow["{\delta'}"', dashed, from=3-1, to=4-1]
	\arrow["i", tail, from=1-1, to=1-2]
	\arrow["p", two heads, from=1-2, to=1-3]
	\arrow["\delta", dashed, from=1-3, to=1-4]
	\arrow["{j'}", color={rgb,255:red,92;green,92;blue,214}, tail, from=2-1, to=2-2]
	\arrow["{q'}", color={rgb,255:red,92;green,92;blue,214}, two heads, from=2-2, to=2-3]
	\arrow["{i'_\ast\delta}", dashed, from=2-3, to=2-4]
	\arrow[Rightarrow, no head, from=1-3, to=2-3]
	\arrow["j"', color={rgb,255:red,92;green,92;blue,214}, tail, from=1-2, to=2-2]
	\arrow["q"', color={rgb,255:red,92;green,92;blue,214}, two heads, from=2-2, to=3-2]
	\arrow["{i_\ast\delta'}"', dashed, from=3-2, to=4-2]
	\arrow[Rightarrow, no head, from=3-1, to=3-2]
\end{tikzcd}\]
where $q^\ast\delta'=-(q')^\ast\delta$ and
$
A\overset{\left[^{\;\;i}_{-i'}\right]}{\infl}B\oplus B'\overset{[j\;j']}{\defl}E\overset{q^\ast\delta'}{\dashrightarrow}
$
is an extriangle.
\end{proposition}

\begin{definition}
An object $P\in\mathscr{C}$ is called \defn{projective} if, for any $X\in\mathscr{C}$, we have $\mathbb{E}(P,X)=0$.
The extriangulated category $\mathscr{C}$ \defn{has enough projectives} if, for any $X\in\mathscr{C}$, there is a deflation
$P\defl X$ with $P$ projective.
There are obvious dual notions for injectives.
\end{definition}

\begin{definition}
An extriangulated category $\mathscr{C}$ is \defn{Frobenius} if it has enough projectives, enough injectives and its injective and projective objects coincide.
The \defn{stable category} $\underline{\mathscr{C}}$ of a Frobenius extriangulated category is the ideal quotient by those morphisms that factor through some projective object.
\end{definition}

Any Frobenius exact category is Frobenius as an extriangulated category.
If $\mathcal{R}$ is a functorially finite rigid subcategory of a 2-Calabi--Yau triangulated category $\mathscr{T}$, then $\mathscr{C} =\mathcal{R}^{\perp_1}$ is a Frobenius extriangulated category.
Indeed, $\mathscr{C}$ is the largest extension-closed full subcategory of $\mathscr{T}$ containing $\mathcal{R}$ and for which each object in $\mathcal{R}$ is both projective and injective.
Because $\mathcal{R}$ is functorially finite, $\mathscr{C}$ has enough projectives and enough injectives.
Therefore, the following proposition unifies \cref{theorem: Happel,theorem: IyamaYoshino} mentionned in \cref{subsection: motivation extricats}.

\begin{proposition}[{\cite[Remark 7.5]{NakaokaPalu}}]
\label{prop:Frobenius}
If $\mathscr{C}$ is a Frobenius extriangulated category, then its stable category is triangulated and the quotient functor is exact.
\end{proposition}

\begin{proof}
The triangulated structure is constructed as follows:
For each $X\in\mathscr{C}$, fix an extriangle
\[
 X\overset{\iota_X}{\infl}I_X\overset{\pi_X}{\defl}\Sigma X \overset{\delta_X}{\dashrightarrow}
\]
with $I_X$ injective.
For any morphism $X\xrightarrow{f}Y$, there is a morphism of extriangles
\[\begin{tikzcd}
	X & {I_X} & { \Sigma X} & {} \\
	Y & {I_Y} & {\Sigma Y} & {}
	\arrow[tail, from=1-1, to=1-2]
	\arrow[two heads, from=1-2, to=1-3]
	\arrow["{\delta_X}", dashed, from=1-3, to=1-4]
	\arrow[tail, from=2-1, to=2-2]
	\arrow[two heads, from=2-2, to=2-3]
	\arrow["{\delta_Y}", dashed, from=2-3, to=2-4]
	\arrow["f"', from=1-1, to=2-1]
	\arrow[from=1-2, to=2-2]
	\arrow["{\sigma_f}", from=1-3, to=2-3]
\end{tikzcd}\]
and letting $\Sigma\underline{f}$ be $\underline{\sigma_f}$ defines an autoequivalence
$\Sigma : \underline{\mathscr{C}}\to\underline{\mathscr{C}}$ which, up to natural isomorphism, does not depend on the choices made above.
For each extension $\sigma\in\mathbb{E}(Z,X)$, there is a morphism of extriangles
\[\begin{tikzcd}
	{(\ast\ast)} & X & Y & Z & {} \\
	& X & {I_X} & {\Sigma X} & {}
	\arrow["i", tail, from=1-2, to=1-3]
	\arrow["p", two heads, from=1-3, to=1-4]
	\arrow[tail, from=2-2, to=2-3]
	\arrow[two heads, from=2-3, to=2-4]
	\arrow["{\delta_X}", dashed, from=2-4, to=2-5]
	\arrow[Rightarrow, no head, from=1-2, to=2-2]
	\arrow[from=1-3, to=2-3]
	\arrow["d", from=1-4, to=2-4]
	\arrow["\delta", dashed, from=1-4, to=1-5]
\end{tikzcd}\]
and the triangles in $\underline{\mathscr{C}}$ are all those diagrams that are isomorphic to some diagram of the form $X\xrightarrow{\underline{i}}Y\xrightarrow{\underline{p}}Z\xrightarrow{\underline{d}}\Sigma X$, with $i,p,d$ as in $(\ast\ast)$ above.

That $\underline{\mathscr{C}}$ is triangulated is an easy adaptation of Happel's proof for exact categories.
An alternative proof using homotopical algebra, which thus only applies to weakly idempotent complete Frobenius extriangulated categories, can be found in~\cite[Corollary 7.4]{NakaokaPalu}.
We thus only prove that the quotient functor is exact.
Define $\Gamma_{Z,X}:\mathbb{E}(Z,X)\to\underline{\mathscr{C}}(Z,\Sigma X)$, by $\Gamma_{Z,X}(\delta) = \underline{d}$ with the notation of $(\ast\ast)$.
The diagram
\[\begin{tikzcd}
	X & {Y'} & Z & {} \\
	X & Y & Z & {} \\
	X & {I_X} & {\Sigma X} & {}
	\arrow["i", tail, from=2-1, to=2-2]
	\arrow["p", two heads, from=2-2, to=2-3]
	\arrow[tail, from=3-1, to=3-2]
	\arrow[two heads, from=3-2, to=3-3]
	\arrow["{\delta_X}", dashed, from=3-3, to=3-4]
	\arrow[Rightarrow, no head, from=2-1, to=3-1]
	\arrow[from=2-2, to=3-2]
	\arrow["d", from=2-3, to=3-3]
	\arrow["\delta", dashed, from=2-3, to=2-4]
	\arrow[Rightarrow, no head, from=1-1, to=2-1]
	\arrow[Rightarrow, no head, from=1-3, to=2-3]
	\arrow["\delta", dashed, from=1-3, to=1-4]
	\arrow["{i'}", tail, from=1-1, to=1-2]
	\arrow["{p'}", from=1-2, to=1-3]
	\arrow["\cong", from=1-2, to=2-2]
\end{tikzcd}\]
shows that $\underline{d}$ does not depend on the choice of a realisation of $\delta$.
It only remains to prove that $\Gamma$ is a natural transformation.
Let $X\xrightarrow{x}X'$ and $Z'\xrightarrow{z}Z$ be morphisms in $\mathscr{C}$.
We have to show that the two diagrams
\[\begin{tikzcd}
	{(1)} & {\mathbb{E}(Z,X)} & {\underline{\mathscr{C}}(Z,\Sigma X)} && {(2)} & {\mathbb{E}(Z,X)} & {\underline{\mathscr{C}}(Z,\Sigma X)} \\
	& {\mathbb{E}(Z,X')} & {\underline{\mathscr{C}}(Z,\Sigma X')} &&& {\mathbb{E}(Z',X)} & {\underline{\mathscr{C}}(Z',\Sigma X)}
	\arrow["{\Gamma_{Z,X}}", from=1-2, to=1-3]
	\arrow["{x_\ast}"', from=1-2, to=2-2]
	\arrow["{\Sigma\underline{x}\circ-}", from=1-3, to=2-3]
	\arrow["{\Gamma_{Z,X'}}", from=2-2, to=2-3]
	\arrow["{\Gamma_{Z,X}}", from=1-6, to=1-7]
	\arrow["{z^\ast}"', from=1-6, to=2-6]
	\arrow["{-\circ\underline{z}}", from=1-7, to=2-7]
	\arrow["{\Gamma_{Z',X}}", from=2-6, to=2-7]
\end{tikzcd}\]
commute.
Choose morphisms of extriangles
\[\begin{tikzcd}
	& X & Y & Z & {} && X & Y & Z & {} \\
	{(a)} & X & {I_X} & {\Sigma X} & {} & {(b)} & {X'} & {Y'} & Z & {} \\
	& {X'} & {I_{X'}} & {\Sigma X'} & {} && {X'} & {I_{X'}} & {\Sigma X'} & {} \\
	&&& X & {I_X} & {\Sigma X} & {} \\
	&& {(c)} & X & E & {Z'} & {} \\
	&&& X & Y & Z & {}
	\arrow[Rightarrow, no head, from=1-2, to=2-2]
	\arrow["x"', from=2-2, to=3-2]
	\arrow[tail, from=1-2, to=1-3]
	\arrow[tail, from=2-2, to=2-3]
	\arrow[tail, from=3-2, to=3-3]
	\arrow[two heads, from=1-3, to=1-4]
	\arrow[two heads, from=2-3, to=2-4]
	\arrow[two heads, from=3-3, to=3-4]
	\arrow[from=1-3, to=2-3]
	\arrow[from=2-3, to=3-3]
	\arrow["{d_1}", from=1-4, to=2-4]
	\arrow["{\sigma_x}", from=2-4, to=3-4]
	\arrow["\delta", dashed, from=1-4, to=1-5]
	\arrow["{\delta_X}", dashed, from=2-4, to=2-5]
	\arrow["{\delta_{X'}}", dashed, from=3-4, to=3-5]
	\arrow["x"', from=1-7, to=2-7]
	\arrow[Rightarrow, no head, from=2-7, to=3-7]
	\arrow[tail, from=1-7, to=1-8]
	\arrow[two heads, from=1-8, to=1-9]
	\arrow["\delta", dashed, from=1-9, to=1-10]
	\arrow[from=1-8, to=2-8]
	\arrow[from=2-8, to=3-8]
	\arrow[Rightarrow, no head, from=1-9, to=2-9]
	\arrow["{d_2}", from=2-9, to=3-9]
	\arrow[tail, from=2-7, to=2-8]
	\arrow[two heads, from=2-8, to=2-9]
	\arrow["{x_\ast\delta}", dashed, from=2-9, to=2-10]
	\arrow[tail, from=3-7, to=3-8]
	\arrow[two heads, from=3-8, to=3-9]
	\arrow["{\delta_{X'}}", dashed, from=3-9, to=3-10]
	\arrow[tail, from=4-4, to=4-5]
	\arrow[tail, from=5-4, to=5-5]
	\arrow[tail, from=6-4, to=6-5]
	\arrow[two heads, from=4-5, to=4-6]
	\arrow[two heads, from=5-5, to=5-6]
	\arrow[two heads, from=6-5, to=6-6]
	\arrow[Rightarrow, no head, from=4-4, to=5-4]
	\arrow[Rightarrow, no head, from=5-4, to=6-4]
	\arrow[from=5-5, to=6-5]
	\arrow[from=5-5, to=4-5]
	\arrow["z", from=5-6, to=6-6]
	\arrow["{d_3}"', from=5-6, to=4-6]
	\arrow["{\delta_X}", dashed, from=4-6, to=4-7]
	\arrow["{z^\ast\delta}", dashed, from=5-6, to=5-7]
	\arrow["\delta", dashed, from=6-6, to=6-7]
\end{tikzcd}\]
By definition, we have $\Gamma_{Z,X'}(x_\ast\delta)=\underline{d_2}$, $\Sigma\underline{x}\circ\Gamma_{Z,X}(\delta) = \underline{\sigma_x\circ d_1}$,
$\Gamma_{Z,X}(\delta)\circ\underline{z} = \underline{d_1\circ z}$ and $\Gamma_{Z',X}(z^\ast\delta)=\underline{d_3}$.
Because the diagrams (a) and (b) are morphisms of extriangles, we have:
\[
 d_2^\ast\delta_{X'} = x_\ast\delta = x_\ast d_1^\ast\delta_X = d_1^\ast x_\ast\delta_X = d_1^\ast\sigma_x^\ast\delta_{X'}=(\sigma_x d_1)^\ast\delta_{X'}.
\]
Hence, $(d_2-\sigma_x d_1)^\ast\delta_{X'} = 0$.
By the long exact sequence of~\cref{prop:extricat long exact sequences}, this implies that $d_2-\sigma_x d_1$ factors through $I_{X'}$.
We thus have 
$\Gamma_{Z,X'}(x_\ast\delta)=\underline{d_2} = \underline{\sigma_x\circ d_1} = \Sigma\underline{x}\circ\Gamma_{Z,X}(\delta)$
 and (1) commutes.
Similarly, it follows from the equalities
\[
 d_3^\ast \delta_X = z^\ast\delta = z^\ast d_1^\ast\delta_X = (d_1 z)^\ast\delta_X
\]
that (2) also commutes.
\end{proof}

We conclude this section by recalling an extriangulated version of the nine lemma.

\begin{proposition}[{\cite[Lemma 5.9]{NakaokaPalu}}]
\label{prop:9lemma}
 Assume that $\mathscr{C}$ is weakly idempotent complete, i.e. that all sections have a kernel.
 Then any commutative square for which two opposite edges are inflations and the other two are deflations can be completed to a diagram made of morphisms of conflations of the form:
\[\begin{tikzcd}
	K & L & \textcolor{rgb,255:red,92;green,92;blue,214}{X} \\
	A & B & C \\
	{A'} & {B'} & {C'}
	\arrow[tail, from=1-1, to=2-1]
	\arrow[tail, from=1-2, to=2-2]
	\arrow[two heads, from=2-1, to=3-1]
	\arrow[two heads, from=2-2, to=3-2]
	\arrow[tail, from=2-1, to=2-2]
	\arrow[tail, from=3-1, to=3-2]
	\arrow[two heads, from=2-2, to=2-3]
	\arrow[two heads, from=3-2, to=3-3]
	\arrow[color={rgb,255:red,92;green,92;blue,214}, tail, from=1-1, to=1-2]
	\arrow[color={rgb,255:red,92;green,92;blue,214}, two heads, from=1-2, to=1-3]
	\arrow[color={rgb,255:red,92;green,92;blue,214}, tail, from=1-3, to=2-3]
	\arrow[color={rgb,255:red,92;green,92;blue,214}, two heads, from=2-3, to=3-3]
\end{tikzcd}\]
\end{proposition}

\begin{remark}
 By~\cref{remark:dualExtricat}, the duals of \cref{prop:htpPO,prop:9lemma} also hold.
\end{remark}

\section{Examples of extriangulated categories}

\subsection{Some examples arising in representation theory}

\subsubsection{Exact categories}

Let $\mathscr{E}$ be an exact category such that, for any $X,Z\in\mathscr{E}$, $\operatorname{Ext}^1(Z,X)$ is a set.
This is true for example if $\mathscr{E}$ is small, or if it has enough injectives or enough projectives.
We let $\mathbb{E}$ be the bifunctor $\operatorname{Ext}^1$, and define
$X \overset{i}{\infl} Y \overset{p}{\defl} Z \overset{\delta}{\dashrightarrow}$
to be an extriangle if and only if $(i,p)$ is a conflation in $\mathscr{E}$ whose equivalence class in $\operatorname{Ext}^1(Z,X)$ is $\delta$.
This gives $\mathscr{E}$ the structure of an extriangulated category whose conflations are precisely those of the exact structure on $\mathscr{E}$.

Conversely, it is shown in~\cite[Corollary 3.18]{NakaokaPalu} that if $(\mathscr{C},\mathbb{E},\mathfrak{s})$ is an extriangulated category all of whose inflations are monomorphisms and all of whose deflations are epimorphisms, then its conflations form an exact structure on $\mathscr{C}$.

\subsubsection{Triangulated categories}
\label{subsubsection:tricats}

Let $(\mathscr{T},\Sigma,\Delta)$ be a triangulated category.
For all $X,Z\in\mathscr{T}$, let $\mathbb{E}(Z,X)=\mathscr{T}(Z,\Sigma X)$ and define
$X \overset{i}{\infl} Y \overset{p}{\defl} Z \overset{\delta}{\dashrightarrow}$
to be an extriangle if and only if $X \xrightarrow{i} Y \xrightarrow{p}Z \xrightarrow{\delta} \Sigma X$ is a triangle in $\Delta$.
This makes $\mathscr{T}$ into an extriangulated category whose extriangles canonically identify with the triangles in $\Delta$.

Conversely, any extriangulated category $\mathscr{C}$ for which $\mathbb{E}$ is given by 
$\mathscr{C}(-,\Sigma-)$ for some autoequivalence $\Sigma:\mathscr{C}\to\mathscr{C}$ is triangulated, with triangles given by the extriangles as above~\cite[Proposition 3.22]{NakaokaPalu}.

\subsubsection{Extension-closed subcategories}
\label{sssection:ext-closed}

If $\mathscr{T}$ is a triangulated category and if $\mathscr{C}$ is a full, extension-closed subcategory, then one can endow $\mathscr{C}$ with the bifunctor $\mathbb{E}=\mathscr{T}(-,\Sigma-)|_{\mathscr{C}^\text{op}\times\mathscr{C}}$ and define extriangles as in \cref{subsubsection:tricats}.
Then $\mathscr{C}$ becomes extriangulated.
Some form of converse holds for topological extriangulated categories (\cite{Klemenc}, see also~\cite{Chen-PhD} for algebraic extriangulated categories, or \cref{subsection:enhancements}).

An important example is the homotopy category of two-term complexes:
If $\Lambda$ is an Artin algebra, the full subcategory $K^{[-1,0]}(\operatorname{proj}\Lambda)$ of the perfect derived category $K^b(\operatorname{proj}\Lambda)$ consisting of complexes concentrated in degree -1 and 0 (with cohomological conventions) is extension-closed, hence extriangulated.

More generally, any extriangulated structure restricts to extension-closed, full subcategories by~\cite[Remark 2.18]{NakaokaPalu}.

\subsubsection{Relative structures}
\label{sssection:RelativeStructures}

Let $(\mathscr{C},\mathbb{E},\mathfrak{s})$ be an extriangulated category.

\begin{definition}[{\cite[Proposition 3.16]{HerschendLiuNakaoka-I}}]
 A biadditive subfunctor $\mathbb{F}\subseteq\mathbb{E}$ is \defn{closed} if $(\mathscr{C},\mathbb{F},\mathfrak{s}|_\mathbb{F})$ is extriangulated.
\end{definition}
If $\mathbb{F}$ is a closed subfunctor of $\mathbb{E}$, then $(\mathscr{C},\mathbb{F},\mathfrak{s}|_\mathbb{F})$ is an extriangulated subcategory of $(\mathscr{C},\mathbb{E},\mathfrak{s})$.
If $\mathbb{F}$ is any biadditive subfunctor of $\mathbb{E}$, then a conflation is called an $\mathfrak{s}|_\mathbb{F}$-conflation if it realises an element $\delta\in\mathbb{F}(Z,X)$, for some $X,Z\in\mathscr{C}$.

\begin{proposition}[{\cite[Proposition 3.11 \& Lemma 3.15]{HerschendLiuNakaoka-I}}]
 Let $\mathbb{F}\subseteq\mathbb{E}$ be a biadditive subfunctor.
 Then $\mathbb{F}$ is closed if and only if $\mathfrak{s}|_\mathbb{F}$-inflations are stable under composition if and only if $\mathfrak{s}|_\mathbb{F}$-deflations are stable under composition.
\end{proposition}

\begin{example}
\label{ex:relative}
Let $\mathscr{T}$ be a $\mathbb{K}$-linear, Hom-finite, 2-Calabi--Yau triangulated category with a cluster tilting object $T$, and let $\mathbb{E}=\mathscr{T}(-,\Sigma-)$.
Define $\mathbb{F}=[\Sigma T](-,\Sigma-)$, where $[\Sigma T]$ denotes the ideal of morphisms factoring through $\operatorname{add}\Sigma T$.
Then $\mathfrak{s}|_\mathbb{F}$-deflations are stable under composition because $\mathscr{T}(T,f)=0$ if and only if $f\in[\Sigma T]$.
Thus $\mathbb{F}$ is a closed subfunctor and $(\mathscr{T},\mathbb{F},\mathfrak{s}|_\mathbb{F})$ is extriangulated.
This extriangulated structure will appear in \cref{section:g-vectors} where it is used to categorify $g$-vectors of cluster algebras.
\end{example}

\begin{proposition}[{\cite[Proposition 3.19]{HerschendLiuNakaoka-I}}]
 Let $\mathcal{R}$ be a full additive subcategory of $\mathscr{C}$, and define 
 $\mathbb{E}_\mathcal{R}(Z,X) = \{\delta\in\mathbb{E}(Z,X)\;|\;\forall\,R\in\mathcal{R}, (\delta_\sharp)_R=0\}$.
 Then $\mathbb{E}_\mathcal{R}$ is a closed subfunctor of $\mathbb{E}$ and gives the biggest extriangulated substructure for which all objects in $\mathcal{R}$ are projective.
 There is a dual version, denoted by $\mathbb{E}^\mathcal{R}$.
\end{proposition}

The closed subfunctor $\mathbb{F}$ of~\cref{ex:relative} coincides with $\mathbb{E}_{\mathcal{R}}$, where $\mathcal{R}=\operatorname{add}T$.

\begin{example}
Let $\mathcal{R}$ be a rigid, full, additive subcategory of some triangulated category $\mathscr{T}$.
Then, $\mathcal{R}\ast\Sigma\mathcal{R}$ might not be extension-closed in $\mathscr{T}$. However, it is extension-closed in $(\mathscr{T},\mathbb{E}_\mathcal{R})$, whence also extriangulated.
\end{example}

\subsubsection{Quotients by projectives-injectives}

Let $(\mathscr{C},\mathbb{E},\mathfrak{s})$ be an extriangulated category, and let $\mathcal{Q}$ be a full subcategory each of whose object is projective-injective.
Because $\mathbb{E}$ vanishes on $\mathcal{Q}^\text{op}\times\mathscr{C}$ and on $\mathscr{C}^\text{op}\times\mathcal{Q}$, it induces a bifunctor $\mathbb{E}:(\mathscr{C}/[\mathcal{Q}])^\text{op}\times(\mathscr{C}/[\mathcal{Q}])\to Ab$.
For $X,Z\in\mathscr{C}/[\mathcal{Q}]$ and $\delta\in\mathbb{E}(Z,X)$ let $\mathfrak{t}(\delta) = [X\xrightarrow{\underline{i}} Y\xrightarrow{\underline{p}}Z]$, where
$\mathfrak{s}(\delta) = [X\xrightarrow{i}Y\xrightarrow{p}Z]$ and $\underline{f}$ denotes the class of a morphism $f\in\mathscr{C}$ in $\mathscr{C}/[\mathcal{Q}]$.
Then $(\mathscr{C}/[\mathcal{Q}],\mathbb{E},\mathfrak{t})$ is extriangulated~\cite[Proposition 3.30]{NakaokaPalu}.

This construction is useful for additive categorification of cluster algebras with coefficients.
Indeed, acyclic cluster algebras with principal coefficients cannot be categorified by Frobenius exact categories (see \cite[Theorem 5.6]{FuKeller} and \cite[Proposition 2.3]{Pressland} for precise statements).
However, Matthew Pressland constructs a Frobenius extriangulated category, obtained by killing half of the projective-injectives of a Frobenius exact category and proves that it categorifies acyclic cluster algebras with principal coefficients~\cite[Corollary 1]{Pressland}.

\subsubsection{Categories of Cohen--Macaulay dg modules}

Let $A$ be dg algebra over a field, and assume that $A$ is connective ($\operatorname{H}^{>0}(A)=0$), proper (its cohomology is of finite total dimension) and Gorenstein (the thick subcategory of the derived category $\mathcal{D}A$ generated by the dual of $A$ coincide with the perfect derived category $\operatorname{per}A$ of $A$).

\begin{definition}[{\cite[Definition 2.1]{Jin}}]
 A dg $A$-module $M$ is \defn{Cohen-Macaulay} if $M\in\mathcal{D}^\text{b}A$, $\operatorname{H}^{>0}(M)=0$ and $\operatorname{Hom}_{\mathcal{D}A}(M,A[i])=0$ for all $i>0$.
\end{definition}

\begin{theorem}[{\cite[Theorem 2.4]{Jin}}]
 The category $\operatorname{CM}A$ of Cohen-Macaulay dg-modules is a Frobenius extriangulated category with subcategory of projective-injectives $\operatorname{add}A$.
 Moreover, there is a triangle equivalence
 $\underline{\operatorname{CM}}A \simeq \mathcal{D}^\text{b}A/\operatorname{per}A.$
\end{theorem}

\subsubsection{Higgs categories}

In \cite{Wu}, Yilin Wu introduces the relative cluster category $\mathcal{C}_n(A,B)$ associated with a morphism $f:B\to A$ of dg algebras satisfying some technical conditions (involving the positive integer $n$).
It is related to Amiot--Guo's cluster categories via the exact sequence of triangulated categories
\[
 0\to\operatorname{per}(eAe) \to \mathcal{C}_n(A,B) \to \mathcal{C}_n(\overline{A})\to 0
\]
where $e=f(1_B)$, $\overline{A}$ is the homotopy cofiber of $f$ and $\mathcal{C}_n(\overline{A})$ is Amiot--Guo's generalised $n$-Calabi--Yau cluster category.

\begin{definition}[{\cite[Theorem 5.42]{Wu}}]
 The Higgs category $\mathscr{H}$ is the full subcategory $^{\perp_{>0}}(eA)\cap(eA)^{\perp_{>0}}$ of $\mathcal{C}_n(A,B)$.
\end{definition}

\begin{theorem}[{\cite[Theorem 5.46]{Wu}}]
The Higgs category $\mathscr{H}$ is a Frobenius extriangulated category with subcategory of projective-injectives $\operatorname{add}(eA)$.
Moreover, its stable category is triangle equivalent to $\mathcal{C}_n(\overline{A})$.
\end{theorem}

\subsection{Enhancements}
\label{subsection:enhancements}

We briefly mention two enhancements for extriangulated categories that are known so far.
They would deserve more details and we refer the interested reader to \cite{Barwick-exact,Klemenc,BorveTrygsland} for exact $\infty$-categories and \cite{Chen-PhD} for exact dg categories.

\subsubsection{Topological extriangulated categories}

Clark Barwick introduced an analogue of exact categories for $\infty$-categories in \cite{Barwick-exact} (see also \cite{DyckerhoffKapranov} for the non-additive case) by interpreting the axioms in an infinity-categorical meaning.
For example, kernel-cokernel pairs should be interpreted as fiber-cofiber sequences.

Exact $\infty$-categories can be seen as a form of enhancement for extriangulated categories:

\begin{theorem}[{\cite{NakaokaPalu2,Klemenc}}]
Let $\mathscr{E}$ be a small exact $\infty$-category.
Then its homotopy category is naturally extriangulated.
\end{theorem}

An extriangulated category which is equivalent to the homotopy category of an exact $\infty$-category is called a \defn{topological extriangulated category}.
The following result is an $\infty$-categorical version of the Gabriel--Quillen embedding theorem.

\begin{theorem}[{\cite{Klemenc}}]
\label{thm:Klemenc}
 Let $\mathscr{E}$ be a small exact $\infty$-category.
 Then there is a stable $\infty$-category $\mathcal{H}^\text{st}(\mathscr{E})$ and a fully faithful functor $\mathscr{E}\to\mathcal{H}^\text{st}(\mathscr{E})$ exhibiting (the essential image of) $\mathscr{E}$ as a full, extension-closed subcategory.
 Moreover, a diagram in $\mathscr{E}$ is a conflation if and only if its image in $\mathcal{H}^\text{st}(\mathscr{E})$ is a fiber-cofiber sequence.
\end{theorem}

In fact, $\mathcal{H}^\text{st}$ defines a functor that is left adjoint to the inclusion functor of stable $\infty$-categories into exact $\infty$-categories, and the functor $\mathscr{E}\to\mathcal{H}^\text{st}(\mathscr{E})$ is the unit of the adjunction.
The stable hull enjoys some universal property: For any stable $\infty$-category $\mathscr{C}$, the unit induces an equivalence between the $\infty$-categories of exact functors $\operatorname{Fun}^\text{ex}(\mathcal{H}^\text{st}(\mathscr{E}),\mathscr{C})$ and $\operatorname{Fun}^\text{ex}(\mathscr{E},\mathscr{C})$.
 
\begin{corollary}
 Any topological extriangulated category is extension-closed in some triangulated category.
\end{corollary}

\begin{remark}
 In \cite{BorveTrygsland}, Erlend Due B{\o}rve and Paul Trygsland define an $\Omega$-spectrum of higher extensions for exact $\infty$-categories and show that it enhances higher extensions of extriangulated categories.
\end{remark}

\subsubsection{Algebraic extriangulated categories}

An algebraic version of \cref{thm:Klemenc} was proven in Xiaofa Chen's PhD thesis~\cite{Chen-PhD}.
He defines exact dg categories by interpreting the axioms of an exact category in the dg setting.
In particular, conflations do not compose to zero, but come equipped with a specified null-homotopy (see \cite[Definition 3.49 and Definition 5.1]{Chen-PhD} for more details).
Exact dg categories are algebraic analogues of exact $\infty$-categories and enhance extriangulated categories.

\begin{theorem}[{\cite[Theorem 6.19]{Chen-PhD}}]
Let $\mathscr{A}$ be an exact dg category.
Then its homotopy category $\operatorname{H}^0\!\mathscr{A}$ is canonically extriangulated.
\end{theorem}

An extriangulated category which is equivalent to the homotopy category of an exact dg category is called an \defn{algebraic extriangulated category}.

\begin{proposition}[{\cite[Definition-Proposition 6.20]{Chen-PhD}}]
 For an extriangulated category $\mathscr{C}$, the following are equivalent:
 \begin{enumerate}[(i)]
  \item It is an algebraic extriangulated category.
  \item It is equivalent to  an extension-closed, full subcategory of an algebraic triangulated category.
  \item It is the quotient of an exact category by some class of projective-injective objects.
 \end{enumerate}
\end{proposition}

The proof of the first implication is based on:
\begin{theorem}[{\cite[Theorem 6.1]{Chen-PhD}}]
 Let $\mathscr{A}$ be a connective exact dg category.
 Then there is a pretriangulated dg category $\mathcal{D}^b_{dg}(\mathscr{A})$ and an exact morphism $\mathscr{A}\to\mathcal{D}^b_{dg}(\mathscr{A})$ (in the homotopy category of the category of small dg categories for the Dwyer--Kan model structure) inducing a quasi-equivalence between $\tau_{\leq0}\mathscr{A}$ and $\tau_{\leq0}\mathscr{D}$, for some extension-closed dg subcategory $\mathscr{D}$ of $\mathcal{D}^b_{dg}(\mathscr{A})$.
\end{theorem}

A consequence of the following result is that any algebraic extriangulated category is topological extriangulated:
\begin{theorem}[{\cite[Theorem 5.26]{Chen-PhD}}]
 Let $\mathscr{A}$ be an exact dg category.
 Then its dg nerve $N_{dg}(\mathscr{A})$ inherits the structure of an exact $\infty$-category.
\end{theorem}

\subsection{Localisation of extriangulated categories}

A theory of localisation for extriangulated categories is developed in~\cite{NakaokaOgawaSakai}.
Given an extriangulated category $(\mathscr{C},\mathbb{E},\mathfrak{s})$, and a class $\mathcal{W}$ of morphisms in $\mathscr{C}$ to be inverted, Theorem~3.5 in \cite{NakaokaOgawaSakai} gives explicit conditions so that the localisation $\mathscr{C}[\mathcal{W}^{-1}]$ inherits an extriangulated structure from $\mathscr{C}$.
Its extriangles are certain equivalence classes of diagrams of the form
\[
A\xrightarrow{w} X \overset{i}{\infl} Y \overset{p}{\defl} Z \xleftarrow{w'} C \overset
{\delta}{\dashrightarrow},
\]
where $w,w'\in\mathcal{W}$, and $X \overset{i}{\infl} Y \overset{p}{\defl} Z \overset
{\delta}{\dashrightarrow}$ is an extriangle in $\mathscr{C}$.

The following specific case applies to many localisations arising in practice.
\begin{definition}
 A full subcategory $\mathcal{N}$ of $\mathscr{C}$ is \defn{thick} if it is closed under summands and isomorphisms and if, for any conflation $X\infl Y\defl Z$ with any two of $X,Y,Z$ in $\mathcal{N}$, the third also belongs to $\mathcal{N}$.
 A thick subcategory $\mathcal{N}$ is \defn{biresolving} if every object $X\in\mathscr{C}$ admits a deflation $N\defl X$ and an inflation $X\infl N'$ with $N,N'\in\mathcal{N}$.
 A thick subcategory is \defn{percolating} if any morphism that factors through $\mathcal{N}$ can be factored as a deflation with codomain in $\mathcal{N}$ followed by an inflation.
\end{definition}

Fix a thick subcategory $\mathcal{N}$ of $\mathscr{C}$ and define $\mathcal{W}_{\mathcal{N}}$ to be the class of morphisms in $\mathscr{C}$ that are finite compositions of inflations with cones in $\mathcal{N}$ and deflations with cocones in $\mathcal{N}$.

\begin{theorem}[{\cite[Corollary 4.27]{NakaokaOgawaSakai}}]
 If the thick subcategory is biresolving, then the localisation $\mathscr{C}[\mathcal{W_{\mathcal{N}}}^{-1}]$ is triangulated.
\end{theorem}
This result generalises Verdier quotients for triangulated categories and \cite[Theorem 5]{Rump} for exact categories.
A version of this result at the level of exact $\infty$-categories follows from the work in progress~\cite{JassoKvammePaluWalde}, mentionned in \cref{rk:JKPW}.

\begin{theorem}[{\cite[Corollary 4.27]{NakaokaOgawaSakai}}]
\label{thm:percolating}
 If the thick subcategory is percolating (and under some mild condition which is always satisfied if $\mathscr{C}$ is weakly idempotent complete, or if it is exact or triangulated), then the localisation $\mathscr{C}[\mathcal{W_{\mathcal{N}}}^{-1}]$ is extriangulated.
\end{theorem}
Moreover, when $\mathcal{N}$ is percolating, \cite[Corollary 4.42]{NakaokaOgawaSakai} gives sufficient condition for the localisation to be exact or even abelian.
\cref{thm:percolating} thus generalises Serre quotients for abelian categories and \cite[Theorem 8.1]{HenrardVanRoosmalen-localisation}, attributed to Manuel Cardenas.

\section{Hovey's correspondence}
\label{section:Hovey}

Hovey's correspondence~\cite{Hovey-CPmodelRT,Hovey-CPmodel} is a device for constructing model category structures on abelian categories.
It is inspired by the somewhat canonical model category structure on a Frobenius category, but with two cotorsion pairs mimicking the role played by the projectives and the injectives.
A similar, but weaker correspondence, due to Apostolos Beligiannis and Idun Reiten and involving only one cotorsion pair appeared independently in~\cite[Theorem 4.2]{BeligiannisReiten} (we note however, that their Theorem~3.5 suggests that some stronger statement should hold).
Hovey's correspondence was generalised to exact categories by James Gillespie~\cite{Gillespie-exact,Gillespie-HoveyTriple,Gillespie-Hereditary}.
He also discovered that this correspondence was a powerful tool for constructing model structures on categories of complexes, with the derived category as its homotopy category.
The strategy of passing from a cotorsion pair on an exact category to a model structure on the category of complexes has been considered, and beautifully generalised, by Henrik Holm and Peter J{\o}rgensen~\cite{HolmJorgensen-Model}. In that article, the category of complexes is thought of as a category of representations of the $A^\infty_\infty$ quiver with radical square zero relations. This point of view allow them to generalise Gillespie's result to categories of representations of self-injective bound quivers, thus including categories of $n$-term complexes, of periodic complexes and many more.
We also note the use of Gillespie's result in order to define the stable category of an arbitrary ring in \cite{BravoGillespieHovey-StableModule}

Hovey's correpondence has also been generalised in a different direction: the proof generalises to triangulated categories, with short exact sequences being replaced by triangles \cite{Yang-TriangulatedHovey}.
Not so surprisingly, it also generalises to extriangulated categories.

\subsection{Exact model structures}

A model structure is equivalently described as two ``intertwined'' weak factorisation systems.

\begin{definition}
 A \defn{weak factorisation system} on a category $\mathscr{C}$ is a pair $(\mathcal{L},\mathcal{R})$ of classes of morphisms in $\mathscr{C}$ satisfying:
 \begin{enumerate}
  \item The classes $\mathcal{L}$ and $\mathcal{R}$ are stable under retracts;
  \item Any morphism $f \in \mathscr{C}$ factorises as $f=r\circ l$, for some $r\in\mathcal{R}$ and $l\in\mathcal{L}$;
  \item The classes $\mathcal{L}$ and $\mathcal{R}$ are weakly orthogonal in the sense that any commutative square
\[\begin{tikzcd}
	A & X \\
	B & Y
	\arrow["l"', from=1-1, to=2-1]
	\arrow["\alpha", from=1-1, to=1-2]
	\arrow["r", from=1-2, to=2-2]
	\arrow["\beta"', from=2-1, to=2-2]
	\arrow["h", dashed, from=2-1, to=1-2]
\end{tikzcd}\]
with $l\in\mathcal{L}$ and $r\in\mathcal{R}$, admits a lifting $h$ making both triangles commute.
 \end{enumerate}
\end{definition}

\begin{notation}
 We write $l\,\square\, r$ if, for any morphisms $\alpha, \beta$ such that the square above commutes, there is a lifting $h$. We also write $\mathcal{L}\,\square\, \mathcal{R}$ when $l\,\square\, r$ for any $(l,r)\in\mathcal{L}\times\mathcal{R}$.
\end{notation}

The following lemma shows a strong relation between lifting properties and $\operatorname{Ext}^1$-vanishing conditions (see~\cite[Lemma 5.6]{Stovicek-ExactModel} for a converse, and more precise, statement).
\begin{lemma}
 Let $X\stackrel{i}{\infl} Y\defl Z$, $A\infl B\stackrel{p}{\defl} C$ be two conflations in an extriangulated category.
 If $\mathbb{E}(Z,A)=0$, then $i$ has the left lifting property with respect to $p$.
\end{lemma}

\begin{proof}
It follows from a straightforward diagram chase in the diagram:
\[\begin{tikzcd}
	&&& {\mathbb{E}(Z,A)} \\
	{(Y,A)} & {(Y,B)} & {(Y,C)} & {\mathbb{E}(Y,A)} \\
	{(X,A)} & {(X,B)} & {(X,C)} & {\mathbb{E}(X,A)} \\
	{\mathbb{E}(Z,A)}
	\arrow[from=2-1, to=2-2]
	\arrow[from=2-2, to=2-3]
	\arrow[from=2-3, to=2-4]
	\arrow[from=1-4, to=2-4]
	\arrow[hook, from=2-4, to=3-4]
	\arrow[two heads, from=2-1, to=3-1]
	\arrow[from=3-1, to=4-1]
	\arrow[from=3-1, to=3-2]
	\arrow[from=3-2, to=3-3]
	\arrow[from=3-3, to=3-4]
	\arrow[from=2-2, to=3-2]
	\arrow[from=2-3, to=3-3]
\end{tikzcd}\]
\end{proof}

\begin{remark}
In a work in progress with Peter J{\o}rgensen, we proved a version adapted to higher homological algebra: Let $\mathscr{C}$ be an $n$-angulated category with suspension functor $\Sigma$, and let $A\overset{f}{\to}B\to C^\bullet \to \Sigma A$ and $X\overset{g}{\to}Y\to Z^\bullet\to\Sigma X$ be two $n$-angles. Then $f\,\square\,g$ if and only if any morphism of complexes $C^\bullet\to Z^\bullet$ is null-homotopic.
\end{remark}

\begin{definition}
 A \defn{complete cotorsion pair} in an extriangulated category $\mathscr{C}$ is a pair $(\mathcal{U},\mathcal{V})$ of full subcategories such that, for any $X\in\mathscr{C}$, the following holds:
 \begin{enumerate}
  \item We have $X\in\mathcal{V}$ if and only if for any $ U\in\mathcal{U}$, $\mathbb{E}(U,X)=0$.
  \item We have $X\in\mathcal{U}$ if and only if for any $V\in\mathcal{V}$, $\mathbb{E}(X,V)=0$.
  \item There are conflations $V_X\infl U_X\defl X$ and $X\infl V^X\defl U^X$ with $U_X,U^X\in\mathcal{U}$ and $V_X,V^X\in\mathcal{V}$.
 \end{enumerate}
\end{definition}

The following is mostly a consequence of the previous lemma.
\begin{proposition}\label{proposition: CP to wfs}
Let $(\mathcal{U},\mathcal{V})$ be a cotorsion pair on an extriangulated category $\mathscr{C}$. Define $\mathcal{L}$ to be the class of all those inflations whose cokernel belongs to $\mathcal{U}$. Dually, define $\mathcal{R}$ to be the class of all those deflations whose kernel belongs to $\mathcal{V}$.
Then $(\mathcal{L},\mathcal{R})$ is a weak factorisation system.
\end{proposition}

\begin{definition}
\label{def:model}
Let $\mathscr{C}$ be a category with finite products and coproducts. Then a \defn{model structure} on $\mathscr{C}$ can be defined as the data of three classes of morphisms $(\mathcal{F}\!\mathit{ib},\mathcal{C}\!\mathit{of},\mathbb{W})$ such that $(\mathcal{C}\!\mathit{of}\cap\mathbb{W},\mathcal{F}\!\mathit{ib})$ and $(\mathcal{C}\!\mathit{of},\mathcal{F}\!\mathit{ib}\cap\mathbb{W})$ are weak factorisation systems and the class $\mathbb{W}$ satisfies the two-out-of-three condition.
\end{definition}
The fact that $\mathbb{W}$ is automatically closed under retracts was first noticed by Andr\'e Joyal and Myles Tierney when $\mathscr{C}$ either has push-outs or has pull-backs (see also~\cite[Lemma 14.2.5]{MayPonto-MoreConcise}). The general case follows from~\cite{Egger} and the fact that weak equivalences are precisely those morphisms that become isomorphisms in the localisation $\mathscr{C}[\mathbb{W}^{-1}]$. This is explained in detail in Pierre Cagne's PhD Thesis~\cite[Section 2.2]{Cagne-PhD}.

\begin{question}\label{question: Hovey}
Let $(\mathcal{S},\mathcal{T}), (\mathcal{U},\mathcal{V})$ be two cotorsion pairs on an extriangulated category $\mathscr{C}$. We obtain, by \cref{proposition: CP to wfs}, two weak factorisation systems $(w\mathcal{C}\!\mathit{of},\mathcal{F}\!\mathit{ib})$, $(\mathcal{C}\!\mathit{of},w\mathcal{F}\!\mathit{ib})$. Define $\mathbb{W}=w\mathcal{F}\!\mathit{ib}\circ w\mathcal{C}\!\mathit{of}$. Under which conditions on the two cotorsion pairs is $(\mathcal{F}\!\mathit{ib},\mathcal{C}\!\mathit{of},\mathbb{W})$ a model structure on $\mathscr{C}$?
\end{question}

We note that if $(\mathcal{F}\!\mathit{ib},\mathcal{C}\!\mathit{of},\mathbb{W})$ is an exact (as in \cref{def:exact model} below) model structure on $\mathscr{C}$, then an object is cofibrant (resp. fibrant, trivially cofibrant, trivially fibrant) if and only if it belongs to $\mathcal{U}$ (resp. to $\mathcal{T}$, $\mathcal{S}$, $\mathcal{V}$).
Part of the answer to \cref{question: Hovey} is then rather immediate:
First, it is necessary that $\mathcal{S}\subseteq\mathcal{U}$ and $\mathcal{V}\subseteq\mathcal{T}$ (those two conditions are equivalent).
Second, both $\mathcal{U}\cap\mathcal{V}$ and $\mathcal{S}\cap\mathcal{T}$ should be the full subcategory of trivially cofibrant and trivially fibrant objects, hence it is necessary that $\mathcal{U}\cap\mathcal{V}=\mathcal{S}\cap\mathcal{T}$.
The last condition is less obvious, and arises when considering the following question: Is it possible to read the two-out-of-three property for $\mathbb{W}$ directly on the two cotorsion pairs?
Some very specific instance of the two-out-of-three property implies that, for any object $X$ in $\mathscr{C}$, the morphism $0\to X$ is a weak equivalence if and only if so is $X\to 0$.
This easily translates to a condition on $(\mathcal{S},\mathcal{T}), (\mathcal{U},\mathcal{V})$. It turns out that this condition (condition (3) below) implies the two-out-of-three property for $\mathbb{W}$ in full generality.

\begin{definition}
 Two cotorsion pairs $(\mathcal{S},\mathcal{T}), (\mathcal{U},\mathcal{V})$ on an extriangulated category $\mathscr{C}$ are called:
 \begin{enumerate}
  \item \defn{Twin} cotorsion pairs if $\mathcal{S}\subseteq\mathcal{U}$ (or equivalently $\mathcal{V}\subseteq\mathcal{T}$);
  \item \defn{Concentric} twin cotorsion pairs if moreover $\mathcal{U}\cap\mathcal{V}=\mathcal{S}\cap\mathcal{T}$;
  \item \defn{Hovey twin cotorsion pairs} if moreover, for any object $X$, the existence of a conflation $V\infl S\defl X$ with $V\in\mathcal{V}$, $S\in\mathcal{S}$ is equivalent to the existence of a conflation $X\infl V'\defl S'$ with $V'\in\mathcal{V}$, $S'\in\mathcal{S}$.
 \end{enumerate}
\end{definition}

Not all model structures on $\mathscr{C}$ can be expected to come from twin cotorsion pairs (for counter-examples, see the model structures arising from rigid objects in~\cite{JacquetMalo-model}). However, all those that interact nicely with the extriangulated structure do.

\begin{definition}
\label{def:exact model}
 A model structure on $\mathscr{C}$ is called an \defn{exact model structure} if its cofibrations are precisely those inflations whose cone is cofibrant and its fibrations are precisely those deflations whose cocone is fibrant.
\end{definition}

In an exact model structure, the acyclic cofibrations coincide with the inflations having a trivially cofibrant cone, and the acyclic fibrations coincide with the deflations having a trivially fibrant cocone.

\subsection{Hovey's correspondence for extriangulated categories}

In order to prove that the weak equivalences defined by some choice of Hovey twin cotorsion pairs satisfy the two-out-of-three property, some technical assumption is needed: 

\begin{definition}
 An additive category is \defn{weakly idempotent complete} if every section has a cokernel or, equivalently, if every retraction has a kernel.
\end{definition}

We note that abelian categories and triangulated categories are weakly idempotent complete as additive categories.

\begin{proposition}[{see~\cite[Section 7]{Buhler-Exact} for exact categories, \cite[Proposition 2.7]{Klapproth-nExtClosed}, \cite[Section II.1.3]{Tattar-thesis} and \cite[Proposition 3.33]{Msapato} for extriangulated categories}]
 Let $\mathscr{C}$ be an extriangulated category. The following are equivalent:
 \begin{enumerate}
  \item The additive category $\mathscr{C}$ is weakly idempotent complete.
  \item Every section is an inflation.
  \item Every retraction is a deflation.
  \item If a composition $g\circ f$ is an inflation, then $f$ is an inflation.
  \item If a composition $g\circ f$ is a deflation, then $g$ is a deflation.
 \end{enumerate}
\end{proposition}

The following result is originaly due to Mark Hovey~\cite{Hovey-CPmodelRT,Hovey-CPmodel} for abelian categories.
It was proven in different generalities in~\cite{BeligiannisReiten} (specific case using one cotorsion pair only), in~\cite{Gillespie-exact} for exact categories, in~\cite{Yang-TriangulatedHovey} for triangulated categories and in~\cite{NakaokaPalu} for extriangulated categories.

\begin{theorem}[Hovey's correspondence]
\label{thm:extriangulatedHovey}
 Let $\mathscr{C}$ be a weakly idempotent complete extriangulated category.
 Then there is a bijective correspondence between Hovey twin cotorsion pairs in $\mathscr{C}$ and exact model structures on $\mathscr{C}$.
\end{theorem}

The correspondence sends a model structure to the cotorsion pairs $(\mathcal{S},\mathcal{T}), (\mathcal{U},\mathcal{V})$, where the four subcategories are given respectively by the trivially cofibrant, the fibrant, the cofibrant and the trivially fibrant objects.

We note the immediate corollary:
\begin{corollary}\label{corollary: Hovey}
 Let $\mathscr{C}$ be a weakly idempotent complete extriangulated category, and let $(\mathcal{S},\mathcal{T}), (\mathcal{U},\mathcal{V})$ be Hovey twin cotorsion pairs on $\mathscr{C}$.
 Then the inclusion $\mathcal{T}\cap\,\mathcal{U}\to\mathscr{C}$ induces an equivalence of categories
 \[
 (\mathcal{T}\cap\,\mathcal{U}) \,/\, [\mathcal{S}\cap\mathcal{V}] \overset{\simeq}{\longrightarrow} \mathscr{C}[\mathbb{W}^{-1}]
 \]
expressing the localisation of $\mathscr{C}$ at the class $\mathbb{W}$ as an ideal subquotient of $\mathscr{C}$.
\end{corollary}

\begin{remark}
\label{rk:JKPW}
 The work in progress~\cite{JassoKvammePaluWalde} with Gustavo Jasso, Sondre Kvamme and Tashi Walde, considers a version of Hovey's correspondence for fibration categories, in the setting of exact $\infty$-categories.
\end{remark}

\subsection{Triangulated homotopy categories}

\Cref{corollary: Hovey} gives a strategy for studying the structure of the homotopy category $\mathscr{C}[\mathbb{W}^{-1}]$.

\begin{definition}
A cotorsion pair $(\mathcal{X},\mathcal{Y})$ of an extriangulated category $\mathscr{C}$ is called \defn{hereditary} if $\mathcal{X}$ is stable under taking kernels of deflations (between objects in $\mathcal{X}$) and $\mathcal{Y}$ is stable under taking cokernels of inflations (between objects in $\mathcal{Y}$).
\end{definition}

\begin{proposition}[James Gillespie~\cite{Gillespie-Hereditary}]
 Let $\mathscr{E}$ be an exact category with some Hovey twin cotorsion pairs $(\mathcal{S},\mathcal{T}),(\mathcal{U},\mathcal{V})$.
Assume that $(\mathcal{S},\mathcal{T})$ and $(\mathcal{U},\mathcal{V})$ are hereditary.
Then the homotopy category $\mathscr{E}[\mathbb{W}^{-1}]$, for the exact model structure obtained by Hovey's correspondence, is triangulated.
\end{proposition}

\begin{proof}
 When $(\mathcal{S},\mathcal{T})$ and $(\mathcal{U},\mathcal{V})$ are hereditary, the extension-closed subcategory $\mathcal{T}\cap\,\mathcal{U}$ of $\mathscr{E}$ is Frobenius exact, with projective-injective objects $\mathcal{S}\cap\mathcal{V}$.
 Hence $(\mathcal{T}\cap\,\mathcal{U})\, /\, [\mathcal{S}\cap\mathcal{V}]$ is triangulated, and so is $\mathscr{C}[\mathbb{W}^{-1}]$ by \cref{corollary: Hovey}.
\end{proof}

We give a generalisation of this result, which seems to be new already in the case of exact categories.

\begin{theorem}\label{theorem: homotopy triangulated}\emph{\cite[Theorem 6.20]{NakaokaPalu}}
For any weakly idempotent complete extriangulated category $\mathscr{C}$ and any Hovey twin cotorsion pairs $(\mathcal{S},\mathcal{T}),(\mathcal{U},\mathcal{V})$ on $\mathscr{C}$, the associated homotopy category $\mathscr{C}[\mathbb{W}^{-1}]$ is triangulated. 
\end{theorem}

\begin{remark}
Under the assumptions of \cref{theorem: homotopy triangulated}, the model structure on $\mathscr{C}$ is not stable \emph{stricto sensu} because $\mathscr{C}$ has too few limits and colimits.
However, the equivalent of the octahedral axiom gives specific choices of weak bicartesian squares that compensate for the lack of (co)limits.
\end{remark}

A specific case of \cref{theorem: homotopy triangulated} is \cref{prop:Frobenius} that generalises \cref{theorem: Happel,theorem: IyamaYoshino} giving homotopical content to these Theorems.

We conclude this section by remarking that \cref{corollary: Hovey} also generalises the following result, which is nicely motivated in the introduction of~\cite{IyamaYang}.

\begin{theorem}[{\cite[Corollary 6.13]{Nakaoka-simultaneous} and \cite[Theorem 1.1]{IyamaYang}}]
\label{thm:Nakaoka-IyamaYang}
 Let $\mathscr{T}$ be a triangulated category, $\mathscr{S}$ be a thick subcategory and $(\mathcal{A},\mathcal{A}^{\perp_1})$ and $(^{\perp_1}\mathcal{B},\mathcal{B})$ be complete cotorsion pairs in $\mathscr{T}$ with $\mathcal{A},\mathcal{B}\subseteq\mathscr{S}$.
 Assume that $(\mathcal{A},\mathcal{B})$ is a complete cotorsion pair in $\mathscr{S}$.
 Then there is an equivalence of categories
 $\left(\mathcal{A}^{\perp_1}\cap\,^{\perp_1}\mathcal{B}\right)/[\mathcal{A}\cap\mathcal{B}] \xrightarrow{\simeq} \mathscr{T}/\mathscr{S}$.
 \end{theorem}

This result (stated here with cotorsion pairs rather than with torsion pairs as in \cite{IyamaYang}) is equivalent to~\cite[Corollary 6.13]{Nakaoka-simultaneous} and also follows from \cref{corollary: Hovey}.
Indeed, the assumptions of \cref{thm:Nakaoka-IyamaYang} can be equivalently expressed by asking $(\mathcal{A},\mathcal{A}^{\perp_1})$ and $(^{\perp_1}\mathcal{B},\mathcal{B})$ to be Hovey twin cotorsion pairs and letting $\mathscr{S}$ be Cone$(\mathcal{A},\mathcal{B}) = \mathcal{B}\ast\Sigma\mathcal{A}$, which is a thick subcategory of $\mathscr{T}$ by~\cite[Proposition 6.8]{Nakaoka-simultaneous} (this also holds if $\mathscr{T}$ is extriangulated by~\cite[Proposition 5.3]{NakaokaPalu}).

\section{Polytopal realisations of g-vector fans}
\label{section:g-vectors}

\subsection{ABHY's construction}

Motivated by the theory of scattering amplitudes in theoretical physics, Arkani-Hamed, Bai, He and Yan~\cite{Arkani-HamedBaiHeYan} introduced a new realization of the classical associahedron.
Fix a family of positive real numbers $\underline{c}=(c_{i\,j})$, for $2\leq i\leq j-2\leq n+1$, and fix variables (coordinates) $q_{i\,j}$ for $1\leq i \leq j-2\leq n+1$, but $(1,n+3)$.
Write $N=\frac{n(n+3)}{2}$, and consider the affine subspace $E_{\underline{c}}$ of $\rb^N$ given by the following equations:
\[
 q_{i\,j}+q_{i+1\,j+1} = q_{i\,j+1}+q_{i+1\,j}+c_{i+1\,j+1}
\]
for each $1\leq i\leq j-2\leq n$.
Let
\begin{eqnarray*}
 \pi : \rb^N & \longrightarrow & \rb^n \\
 (q_{i\,j}) & \longmapsto & (q_{2,n+3},\ldots,q_{n+1\,n+3})
\end{eqnarray*}
and define $\asso_{\underline{c}} = \pi\left(E_{\underline{c}}\cap \rb_{>0}^N\right)$.

For $n=2$, $E_{\underline{c}}$ is thus given by the system of equations:
\[
E_{\underline{c}} :  \left\{ \begin{array}{ccl}
q_{1\,3}+q_{2\,4} & = & q_{1\,4} + c_{2\,4} \hspace{1cm} (1) \\
q_{1\,4}+q_{2\,5} & = & q_{2\,4} + c_{2\,5} \hspace{1cm} (2) \\
q_{2\,4}+q_{3\,5} & = & q_{2\,5} + c_{3\,5} \hspace{1cm} (3)
\end{array} \right.
\]
which is equivalent to
\[
E_{\underline{c}} :  \left\{ \begin{array}{ccl}
q_{1\,3} & = & c_{2\,4}  + c_{2\,5} - q_{2\,5} \hspace{1cm} (1)+(2) \\
q_{1\,4} & = & c_{2\,5} + c_{3\,5} - q_{3\,5} \hspace{1cm} (2)+(3) \\
q_{2\,4} & = & c_{3\,5} + q_{2\,5} - q_{3\,5} \hspace{1cm} (3)
\end{array} \right.
\]
Intersecting it with the positive orthant and only remembering the last two coordinates, we obtain
\begin{eqnarray*}
 \asso_{\underline{c}} & = & \{(q_{2\,5},q_{3\,5})\in\rb_{>0}^2 \;|\; q_{1\,3} >0, q_{1\,4} >0, q_{2\,4} >0\} \\
 & = & \{(q_{2\,5},q_{3\,5})\in\rb_{>0}^2 \;|\; q_{2\,5} < c_{2\,4}  + c_{2\,5},\; q_{3\,5} < c_{2\,5} + c_{3\,5},\; q_{3\,5} - q_{2\,5} < c_{3\,5} \}.
\end{eqnarray*}

\begin{figure}[h]
\begin{center}
 \capstart
  \includegraphics[scale=.7]{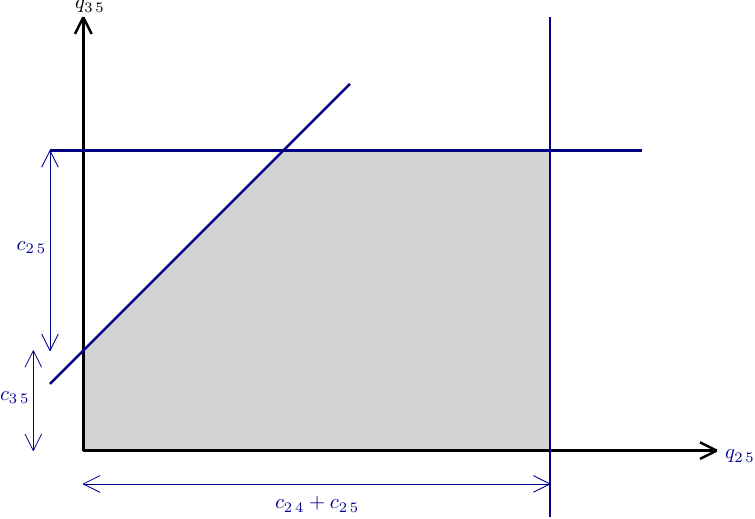}
  \caption{ABHY's construction for $n=2$ gives the associahedron of type $A_2$.}
  \label{fig:Asso_2}
\end{center}
\end{figure}

\begin{theorem}[{\cite{Arkani-HamedBaiHeYan}}]
 For all positive $\underline{c}$, $\asso_{\underline{c}}$ is an associahedron of type $A_n$.
\end{theorem}

\begin{question}
 Is it possible to understand this result by using representation theory?
\end{question}

That is precisely what a team from the LaCIM did, allowing them to generalise the result to other Dykin types.

\subsection{Representation theoretic interpretation}

The team from the LaCIM alluded to above was formed by:
V\'eronique Bazier-Matte, Nathan Chapelier-Laget, Guillaume Douville, Kaveh Mousavand, Hugh Thomas and Emine Y{\i}ld{\i}r{\i}m~\cite{BazierMatteChapelierLagetDouvilleMousavandThomasYildirim}.
This subsection is inspired from Hugh Thomas' talk at the RIMS in 2019.

Let $Q$ be a Dynkin quiver of type $A,D$ or $E$.
Denote by $\mathcal{D}_Q$ the full subcategory $\operatorname{add} (\{M\in\operatorname{mod}\field Q, M \text{indecomposable}\}\cup \{I_j[-1], j\in Q_0\})$  of $\mathcal{D}^{\text{b}}(\field Q)$.
For each $M\in\operatorname{ind}\mathcal{D}_Q$, fix a coordinate $q_M$.
For each $M\in\operatorname{ind}\mathbb{C}Q$ fix some positive real number $c_M>0$.
Let $n=|Q_0|$ and $N=|\operatorname{ind}\mathcal{D}_Q|$.

\begin{example}
 If $Q=1\leftarrow 2$, then $\operatorname{ind}\mathbb{C}Q$ contains the following three isomorphism classes.

$P_1=S_1=\mathbb{C}\leftarrow 0$, $P_2=I_1=\mathbb{C}\stackrel{1}{\leftarrow}\mathbb{C}$, and $S_2=I_2= 0\leftarrow\mathbb{C}$.
The Auslander--Reiten quiver of the subcategory $\mathcal{D}_Q$ is given by:

\[
\adjustbox{scale=.85,center}{%
\begin{tikzcd}[scale=.5]
	& {I_2[-1]} && {I_1} &&&&& {q_{1\,4}} & {c_{2\,5}} & {q_{2\,5}} \\
	{I_1[-1]} && {P_1} && {I_2} &&& {q_{1\,3}} & {c_{2\,4}} & {q_{2\,4}} & {c_{3\,5}} & {q_{3\,5}}
	\arrow[from=2-1, to=1-2]
	\arrow[from=1-2, to=2-3]
	\arrow[from=2-3, to=1-4]
	\arrow[from=1-4, to=2-5]
	\arrow[from=2-8, to=1-9]
	\arrow[from=1-9, to=2-10]
	\arrow[from=2-10, to=1-11]
	\arrow[from=1-11, to=2-12]
\end{tikzcd}
}
\]
\end{example}

For $c=(c_M)_{M\in\operatorname{ind}\mathcal{D}_Q}$, define $\mathbb{E}_c$ to be the affine subspace of $\mathbb{R}^N$ given by the $N-n =|\operatorname{ind}\mathbb{C}Q|$ equations
$q+t=r_1+r_2+r_3+c$ for each mesh
\[
\adjustbox{scale=.85,center}{%
\begin{tikzcd}
	& {r_1} \\
	q & c & t \\
	& {r_2} \\
	& {r_3}
	\arrow[from=2-1, to=1-2]
	\arrow[from=1-2, to=2-3]
	\arrow[from=2-1, to=3-2]
	\arrow[from=3-2, to=2-3]
	\arrow[from=2-1, to=4-2]
	\arrow[from=4-2, to=2-3]
\end{tikzcd}
}
\]
in the Auslander--Reiten quiver of $\mathcal{D}_Q$, with the convention that $r_3$ and $r_2$ are 0 if the mesh has fewer middle terms.
Consider the projection
\begin{eqnarray*}
 \pi : \mathbb{R}^N & \longrightarrow & \mathbb{R}^n \\
 (q_M)_{M\in\operatorname{ind}\mathcal{D}_Q} & \longmapsto & (q_{I_j})_{j\in Q_0}
\end{eqnarray*}
and define $\mathbb{A}_c$ by $\mathbb{A}_c=\pi\left(\mathbb{E}_c\cap \mathbb{R}_{>0}^N\right)$.

\begin{theorem}[Bazier-Matte--Chapelier-Laget--Douville--Mousavand--Thomas--Y{\i}ld{\i}r{\i}m]
 For any positive $c$, $\mathbb{A}_c$ is a generalised associahedron of the same Dynkin type as $Q$.
\end{theorem}

This beautiful approach using representations of quivers has some drawbacks. First, the authors have to work with parameters $c_M$ in $\mathbb{Q}$, and deduce the result over $\mathbb{R}$ by density. Second, they have to work over an acyclic intial seed (this was later lifted in~\cite{BazierMatte-thesis}).
Third, non-simply laced Dynkin types are deduced from simply-laced types by a non-trivial folding technique~\cite{Arkani-HamedHeLam}.
More arguably, one might be tempted to say that their methods gives an interpretation of ABHY's construction allowing for more general versions but does not explain why ABHY's construction works.

\begin{figure}[h]
\begin{center}
 \capstart
  \includegraphics[scale=.7]{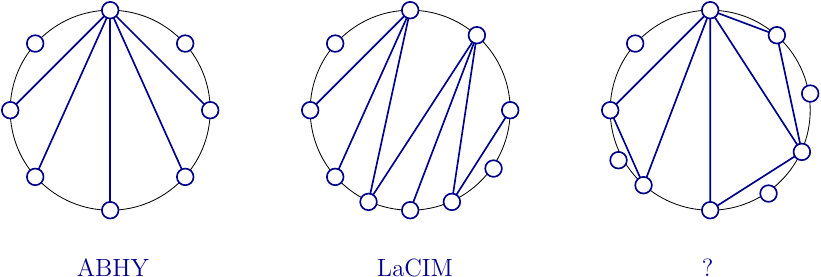}
  \caption{The cases treated by the two teams, in Dynkin type $A$: ABHY considers linear orientation (i.e. fan triangulations), LaCIM considers acyclic initial seeds (i.e. triangulations with no internal triangles). We aim at treating the general case (and at extending to dissections)}
  \label{fig:Triangulations}
\end{center}
\end{figure}

\subsection{Reformulation using McMullen's type cone}

Our aim is to give another interpretation of ABHY's construction that avoids the drawbacks mentionned above, while explaining why this construction is natural.
The tools we will make use of are McMullen's type cone and cluster categories.


\subsubsection{Polyhedral cones}
A subset of~$\mathbb{R}^n$ is a \defn{polyhedral cone} if it is given as the non-negative span of finitely many vectors. A cone is \defn{simplicial} if it is generated by a set of linearly independent vectors.

Let $X \subset \mathbb{R}^n$. A \defn{supporting hyperplane} of $X$ is a hyperplane $H \subset \mathbb{R}^n$ that intersects $X$ and such that $X$ is entirely contained in one of the half-spaces defined by $H$.
The \defn{faces} of a cone $C$ are the intersections of $C$ with its supporting hyperplanes. Faces of dimension one are called \defn{rays} while faces of codimension one are called \defn{facets}.

\subsubsection{Polyhedral fans}
A \defn{polyhedral fan} is a collection~$\mathcal{F}$ of polyhedral cones satisfying the following two conditions:
\begin{itemize}
\item For any cone $C$ in $\mathcal{F}$, each face of $C$ is in $\mathcal{F}$.
\item For any two cones $C,C'\in\mathcal{F}$, $C\cap C'$ is both a face of $C$ and of $C'$.
\end{itemize}
Given a fan $\mathcal{F}$, we will always fix a generator for each ray, and we will denote cones by $\mathbb{R}_{\geq 0} \mathcal{R}$, where $\mathcal{R}$ is the set of the chosen generators of the rays of the cone.

A fan $\mathcal{F}$ is
\begin{itemize}
 \item \defn{simplicial} if all its cones are simplicial;
 \item \defn{complete} if $\mathbb{R}^n = \cup\mathcal{F}$;
 \item \defn{essential} if $\{0\}$ is one of its cones.
\end{itemize}

\subsubsection{Polytopes and their normal fans}
A \defn{polytope} is the convex hull of finitely many points in $\mathbb{R}^n$. It can be equivalently defined as the bounded intersection of finitely many closed affine halfspaces.
Similarly as for cones, we call \defn{faces} of a polytope its intersections with some supporting hyperplane and its \defn{facets} are its codimension one faces.
Its 0-dimensional and 1-dimensional faces are respectively called \defn{vertices} and  \defn{edges}.


The (outer) \defn{normal cone} of a face $F$ of a polytope $P$ is the cone generated by the outer normal vectors of the facets of $P$ containing $F$.
The \defn{normal fan} of $P$ is the collection of the normal cones of all its faces.
If a complete polyhedral fan $\mathcal{F}$ in $\mathbb{R}^n$ is the normal fan of a polytope $P$ of $\mathbb{R}^n$, we say that it is \defn{polytopal}, and $P$ is then called a \defn{polytopal realization} of $\mathcal{F}$.

\begin{figure}[h]
\begin{center}
 \capstart
  \includegraphics[scale=.65]{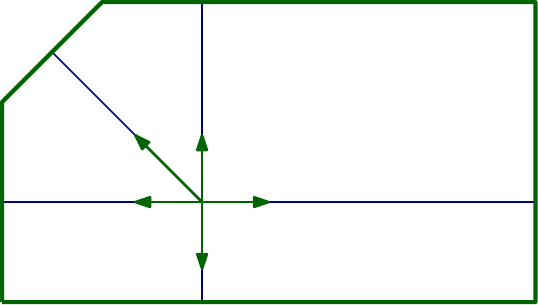} \hfill \includegraphics[scale=.65]{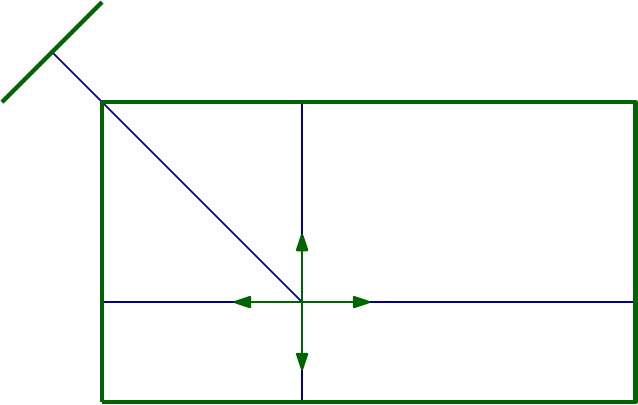}
  \caption{On the left hand side, we picked a good choice of normal hyperplanes and we obtain a realisation of the classical associahedron of type $A_2$. On the right hand side, a bad choice of normal hyperplanes.}
  \label{fig:GoodVSBadRealisations}
\end{center}
\end{figure}

\subsubsection{McMullen's type cone}
We fix an essential, complete, simplicial fan $\mathcal{F}$ in $\mathbb{R}^n$ with $N$ rays arbitrarily ordered.
The type cone of $\mathcal{F}$, introduced in \cite{McMullen-typeCone}, morally parametrizes all polytopal realisations of $\mathcal{F}$.

For any two adjacent maximal cones $\mathbb{R}_{\ge0}\mathcal{R}$ and~$\mathbb{R}_{\ge0}\mathcal{R}'$ of $\mathcal{F}$ with ${\mathcal{R} \smallsetminus \{r\} = \mathcal{R}' \smallsetminus \{r'\}}$, there is, up to scalar multiplication, a unique linear dependency of the form
\[
\alpha r +\alpha' r' = \sum_{s_i\in\mathcal{R}\cap\mathcal{R}'}\alpha_i s_i
\]
where we may and will assume $\alpha,\alpha'>0$. 
We will also make the following choice: $\alpha+\alpha'=2$
(this convention is arbitrary, but naturally arises for $g$-vector fans from categorification of cluster algebras).

\begin{figure}[h]
\begin{center}
 \capstart
  \includegraphics[scale=.65]{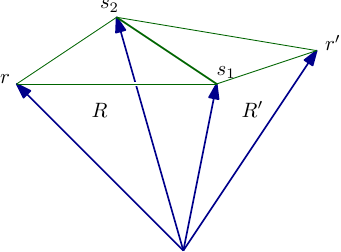}
  \caption{Two adjacent cones in $\mathbb{R}^3$ giving rise to a relation $\alpha r + \alpha' r' = \alpha_1 s_1 + \alpha_2 s_2$.}
  \label{fig:AdjacentCones}
\end{center}
\end{figure}

\begin{definition}
\label{def:typeCone}
The \defn{type cone} of $\mathcal{F}$ is the cone
\[
 \mathbb{TC}(\mathcal{F}) = \left\{h \in \mathbb{R}^N\; | \;\alpha h_r +\alpha' h_{r'} > \sum_{s_i\in\mathcal{R}\cap\mathcal{R}'}\alpha_i h_{s_i}  \; \begin{array}{l} \text{for any adjacent maximal} \\ \text{cones } \mathbb{R}_{\ge0}\mathcal{R} \text{ and } \mathbb{R}_{\ge0}\mathcal{R}' \text{ of } \mathcal{F} \end{array}\right\}.
\]
\end{definition}
We note that this version of the definition comes from~\cite[Lemma 2.1]{ChapotonFominZelevinsky}.

Let $G$ be the $N \times n$-matrix whose rows are the coordinates of the vectors defining the rays of $\mathcal{F}$.
For any height vector $h \in \mathbb{R}^N$, we define the polytope
\[
P_h = \{x \in \mathbb{R}^n \;|\; Gx \le h\},
\]
where $x\le y$ if the inequality holds coordinate-wise.

\begin{lemma}\cite{ChapotonFominZelevinsky}
Let $h\in\mathbb{R}^N$. Then, $P_h$ is a polytopal realization of $\mathcal{F}$ if and only if $h\in\mathbb{TC}(\mathcal{F})$.
\end{lemma}

\subsubsection{From the type cone to ABHY's construction}

The key input from part I of~\cite{PadrolPaluPilaudPlamondon}, is to realise that ABHY's construction is, in good cases, a convenient reformulation of the type cone.

\begin{theorem}\cite{PadrolPaluPilaudPlamondon}
\label{theorem:PPPPpartI}
 Assume that the  type cone $\mathbb{TC}(\mathcal{F})$ of the fan $\mathcal{F}$ is simplicial (equivalently: has $N-n$ facets).
Let $K$ be the $(N-n)\times N$ matrix whose rows are the coordinates of the inner normal vectors of the facets of $\mathbb{TC}(\mathcal{F})$.
Then, for any positive $c\in\mathcal{R}_{>0}^{N-n}$, the polytope
\[
 Q_c = \left\{q\in\mathbb{R}^N \;|\; Kq=c \text{ and } q\ge0 \right\}
\]
is a polytopal realisation of $\mathcal{F}$.
Moreover, the polytopes $Q_c$ describe all polytopal realisations of $\mathcal{F}$.
\end{theorem}

\begin{remark}
Let $h\in\mathbb{R}^N$ be such that $Kh=c$.
 Then the polytope $Q_c$ is obtained from $P_h$, via a standard method, by some well-chosen affine transformation, and the condition $h\in\mathbb{TC}(\mathcal{F})$ is equivalent to $c>0$.
 Going from $P_h$ to $Q_c$ can be thought of as going from the type cone to ABHY's construction.
We thus have an explanation of why ABHY's construction works.
\end{remark}

Our aim is now to apply \cref{theorem:PPPPpartI} to the $g$-vector fan $\mathcal{F}_\Sigma$ of a cluster algebra of finite type with initial seed $\Sigma$ in order to recover ABHY's construction and to generalise it to any initial seed.
In order to do so, we have to show that $\mathcal{F}_\Sigma$ satisfies the unique exchange relation property, that its type cone is simplicial and that the equations defining the polytope $Q_c$ coincide with those appearing in the LaCIM's interpretation of ABHY's construction.
That is where we start using representation theory and where:

\subsection{Cluster categories come into play}

Let $\Delta$ be a Dynkin diagram and let $\mathcal{A}_\Delta$ be the associated cluster algebra of finite type.
Fix any initial seed $\Sigma_0$, and let $\mathcal{F}_{\Sigma_0}$ be the fan given by the $g$-vectors relative to $\Sigma_0$.

\begin{definition}
 The simplicial fan $\mathcal{F}$ has the \defn{unique exchange relation property} if the linear dependencies $\alpha r +\alpha' r' = \sum_{s_i\in\mathcal{R}\cap\mathcal{R}'}\alpha_i s_i
$ only depend on $r$ and $r'$ and not on the choice of the cones $\mathbb{R}_{\ge 0}\mathcal{R}$ and $\mathbb{R}_{\ge 0}\mathcal{R}'$.
\end{definition}

\begin{theorem}\cite{PadrolPaluPilaudPlamondon}
\label{theorem:PPPPpartII}
The $g$-vector fan $\mathcal{F}_{\Sigma_0}$ has the unique exchange relation property and its type cone is simplicial.
Moreover, the associated polytopes $Q_c$ in~\cref{theorem:PPPPpartI} are precisely given by ABHY's construction in a re-interpretation similar to that of Bazier-Matte--Chapelier-Laget--Douville--Mousavand--Thomas--Y{\i}ld{\i}r{\i}m.
\end{theorem}

\noindent {\bf Sketch of proof}

\medskip

\noindent \emph{First step}: Replace the full subcategory $\mathcal{D}_Q$ by the cluster category $\mathscr{C}_\Delta$ of \cite{BuanMarshReinekeReitenTodorov}.

\medskip

It is the orbit category of the bounded derived category of $\Delta$ with respect to the action of the autoequivalence $\tau^{-1}[1]$ (see \cref{fig:ClusterCategory} for an example).
By~\cite{Keller-orbit}, the cluster category $\mathscr{C}_\Delta$ inherits a triangulated structure from the derived category.
Its shift functor we denote by $\Sigma$.
The functor $\tau^{-1}[1]$ is chosen so that, in the cluster category, $\Sigma$ is naturally isomorphic to $\tau$. In other words, $\mathscr{C}_\Delta$ is 2-Calabi--Yau.

\begin{figure}[h]
\begin{center}
 \capstart
  \includegraphics[scale=.6]{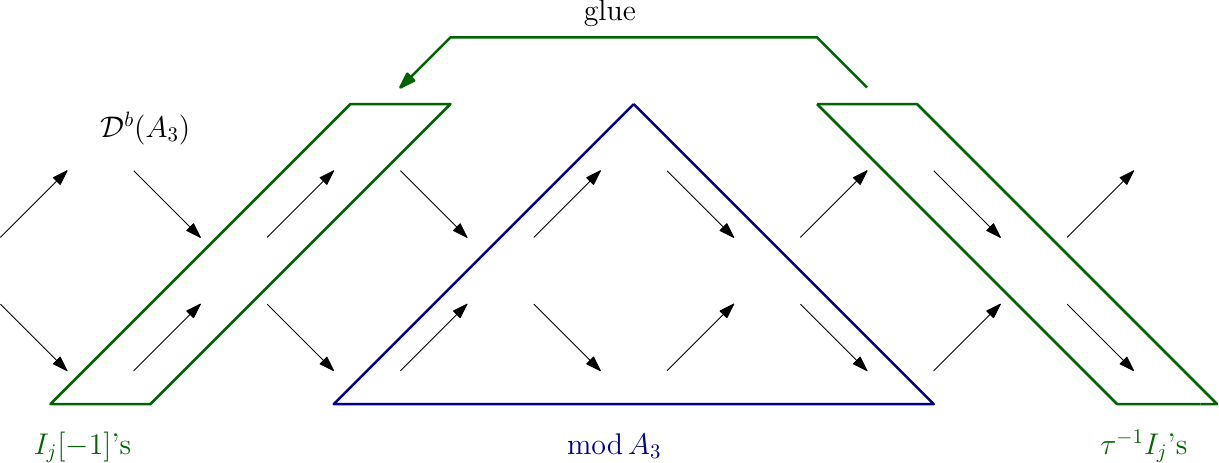}
  \caption{The Auslander--Reiten quiver of the cluster category of type $A_3$.}
  \label{fig:ClusterCategory}
\end{center}
\end{figure}

\noindent \emph{Second step}: Replace the triangulated structure of $\mathscr{C}_\Delta$ by some relative extriangulated structure.

\medskip

More precisely, there is a cluster tilting object $T\in\mathscr{C}_\Delta$ associated with the seed $\Sigma_0$.
Let $\Delta_T$ be the collection of all those triangles $X\to Y\to Z\xrightarrow{\delta}\Sigma X$ fo which the morphism $\delta$ factors through an object isomorphic to $\Sigma T^{\oplus r}$ for some $r$.
Then by \cite[Proposition 3.19]{HerschendLiuNakaoka-I} (see \cref{sssection:RelativeStructures}) $\mathscr{C}_\Delta$ endowed with $\Delta_T$ is an extriangulated category.
Even though this extriangulated structure has the drawback of not being 2-Calabi--Yau anymore, it has some interest in categorifying $g$-vectors:
The Grothendieck group of the triangulated $\mathscr{C}_\Delta$ has torsion (we even have K$_0(A_2)=0$) whereas $g$-vectors naturally live in the Grothendieck group of the extriangulated $\mathscr{C}_\Delta$.
Let 
\[
 \operatorname{K}(\mathscr{C}_\Delta;T)=\operatorname{K}^\text{split}_0(\mathscr{C}_\Delta)/\langle [L]+[N]-[M]\;|\; L\to M\to N\to\Sigma L \text{ in } \Delta_T \rangle
\]
be the Grothendieck group of $\mathscr{C}_\Delta$ endowed with the extriangulated structure given by $\Delta_T$, and let
\[
 \mathfrak{g}: \operatorname{K}^\text{split}_0(\mathscr{C}_\Delta) \to \operatorname{K}_0(\mathscr{C}_\Delta;T)
\]
be the quotient map.

\medskip

\noindent \emph{Third step}: Compute the Grothendieck group K$_0(\mathscr{C}_\Delta;T)$.

\medskip

For any object $X\in\mathscr{C}_\Delta$, there is a triangle $T_1\to T_0\to X\to \Sigma T_1$.
The assignement $\operatorname{ind}_T: X\mapsto [T_0]-[T_1]$ induces a well-defined group homomorphism $\operatorname{K}^\text{split}_0(\mathscr{C}_\Delta)\to \operatorname{K}^\text{split}_0(\operatorname{add}T)$.
One can deduce from~\cite[Proposition 2.2]{Palu-CC} that there is a commutative diagram
\[\begin{tikzcd}
	{\operatorname{K}^\text{split}_0(\mathscr{C}_\Delta)} \\
	\\
	{\operatorname{K}_0(\mathscr{C}_\Delta;T)} && {\operatorname{K}^\text{split}_0(\operatorname{add}T)}
	\arrow["{\mathfrak{g}}"{description}, two heads, from=1-1, to=3-1]
	\arrow["{\operatorname{ind}_T}"{description}, from=1-1, to=3-3]
	\arrow["\varphi"{description}, from=3-1, to=3-3]
\end{tikzcd}\]
where $\varphi$ is a group isomorphism.

\medskip

\noindent \emph{Fourth step}: Study the kernel of the quotient map $\mathfrak{g}$.

Write $T=T_1\oplus\cdots\oplus T_n$ where the $T_i$'s are indecomposable and $T$ is assumed basic.
We prove in \cite[Corollary 3.9]{PadrolPaluPilaudPlamondon} that $\operatorname{Ker}\mathfrak{g}$ is generated, as a monoid, by the relations $[L]+[N]-[M]$, for each Auslander--Reiten triangle $L\to M\to N\to \Sigma L$ in $\mathscr{C}_\Delta$ with $L$ not isomorphic to any $\Sigma T_i$.
Those are precisely the Auslander--Reiten sequences (see~\cite[Proposition 5.10]{IyamaNakaokaPalu}) of the extriangulated category $(\mathscr{C}_\Delta, \Delta_T)$.

\medskip

\noindent \emph{Last step}: The type cone is simplicial

There are $N$ Auslander--Reiten triangles in $\mathscr{C}_\Delta$, hence $N-n$ Auslander--Reiten sequences in  $(\mathscr{C}_\Delta, \Delta_T)$.
One thus deduces from the above that $\mathbb{TC}(\mathscr{F}_\Delta)$ is simplicial.
Moreover, from the form of the Auslander--Reiten sequences, one also deduces that using the type cone construction recovers the ABHY construction in the LaCIM's version for simply laced types. 

\qed

\begin{remark}
The results in \cite{PadrolPaluPilaudPlamondon} on $(\mathscr{C}_\Delta, \Delta_T)$ are proven in greater generality, by considering some hereditary extriangulated categories.
They can thus be applied to other categories such as the categories $K^{[-1,0]}(\operatorname{proj}\Lambda)$ of 2-term complexes of projectives over a finite-dimensional algebra.
It follows that ABHY's construction also works for realising $g$-vector fans of (brick, 2-acyclic) gentle algebras of finite representation type.
Due to their relation with dissections on oriented surfaces, those polytopal realisations turned out to be useful in theoretical physics (see \cite{JagadaleLaddha,Chhatoi-Halohedron,ABJJLM} for instance).
\end{remark}

\section{Mutation}
\label{section:mutation}

This section, based on~\cite{GorskyNakaokaPalu-Mutation}, aims at explaining how extriangulated categories form a convenient setting for studying mutations in representation theory.

\subsection{0-Auslander extriangulated categories}

As illustrated by the previous section, the extriangulated categories $K^{[0,1]}(\operatorname{proj}\Lambda)$ and $(\mathscr{C}_\Delta,\Delta_T)$ may be used to categorify $g$-vector fans.
Both those categories enjoy some nice properties which motivates the following:

\begin{definition}
\label{def:0-Auslander}
 An extriangulated category $\mathscr{C}$ is \defn{0-Auslander} if
 \begin{enumerate}[(a)]
  \item for any $X\in\mathscr{C}$, there is a conflation $P_1\infl P_0\defl X$ with $P_0$ and $P_1$ projective and
  \item for any projective $P\in\mathscr{C}$, there is a conflation $P\infl Q\defl I$ where $Q$ is projective-injective and $I$ is injective.
 \end{enumerate}
\end{definition}

\begin{remark}
 Conditions (a) and (b) can be interpreted as $\operatorname{gl.dim}\mathscr{C} \leq 1\leq \operatorname{dom.dim}\mathscr{C}$, whence the name 0-Auslander.
\end{remark}

\begin{remark}
One easily shows that \cref{def:0-Auslander} is self-dual. Indeed, it is equivalent to
 \begin{enumerate}[(a')]
  \item for any $X\in\mathscr{C}$, there is a conflation $X\infl I_0\defl I_1$ with $I_0$ and $I_1$ injective and
  \item for any injective $I\in\mathscr{C}$, there is a conflation $P\infl Q\defl I$ where $Q$ is projective-injective and $P$ is projective.
 \end{enumerate}
\end{remark}

\begin{example}
\label{ex:0-Auslander}
 The following are examples of 0-Auslander extriangulated categories.
 \begin{enumerate}
  \item[($\alpha$)] The category of finite-dimensional representations $\operatorname{rep}_\mathbb{K}\overrightarrow{A_n}$ of a quiver of type $A_n$ with linear orientation, over a field $\mathbb{K}$.
  If $\mathbb{K}$ is algebraically closed, this is the only non-semisimple example of a 0-Auslander module category~\cite{Tachikawa}.
  \item[($\beta$)] Let $Q$ be a quiver of type $A_n$, with any orientation, and let $\mathbb{K}$ be a field. There is an exact structure $\mathscr{E}$ on $\operatorname{rep}_\mathbb{K}Q$ whose Auslander--Reiten sequences are precisely those of $\operatorname{rep}_\mathbb{K}Q$ that have non-indecomposable middle term.
  Then $\mathscr{E}$ is a 0-Auslander exact category~\cite{BrustleHansonRoySchiffler-AlmostRigid}.
  \item[($\gamma$)] Let $\mathscr{C}$ be a cluster category~\cite{BuanMarshReinekeReitenTodorov}, a generalised cluster category~\cite{Amiot-ClusterCats} or, more generally, any Hom-finite 2-Calabi--Yau triangulated category with some cluster tilting object $T$.
  There is a maximal extriangulated structure $\mathbb{E}^T$ on $\mathscr{C}$ making $T$ projective~\cite[Proposition 3.19]{HerschendLiuNakaoka-I}.
  Then $(\mathscr{C},\mathbb{E}^T)$ is 0-Auslander.
  \item[($\delta$)] Let $\mathcal{A}$ be an additive category (our motivating example being $\mathcal{A}=\operatorname{proj}\Lambda$, for an Artin algebra $\Lambda$).
  Then the category of two-term complexes $C^{[-1,0]}(\mathcal{A})$ and the homotopy category two-term complexes $K^{[-1,0]}(\mathcal{A})$ are 0-Auslander~\cite{BautistaSalorioZuazua}.
  \item[($\varepsilon$)] If $\Lambda$ is a finite-dimensional gentle algebra over an algebraically closed field, there is an associated oriented surface with dissection~\cite{BaurCoelhoSimoes,OpperPlamondonSchroll,PaluPilaudPlamondon-surfaces}.
  Some specific arcs on the surface, called accordions, play a key role in describing combinatorially the $\tau$-tilting theory of the gentle algebra $\Lambda$.
  There is an exact category $\mathscr{W}$ whose indecomposable objects are the accordions and whose non-split extensions categorify crossings of accordions~\cite[Section 6]{GorskyNakaokaPalu-Mutation}.
  The category $\mathscr{W}$ is 0-Auslander. See \cref{fig:ARquiverAccordions} for an example.
  \item[($\zeta$)] If $(\mathcal{A},\mathcal{B})$ is a co-$t$-structure~\cite{Pauksztello,Bondarko} in a triangulated category $\mathscr{T}$, then its extended coheart $\mathscr{C}=(\Sigma^2\mathcal{A})\cap\mathcal{B}$ is a 0-Auslander extriangulated category.
 \end{enumerate}
\end{example}

\begin{remark}
 In the examples $(\mathscr{C},\mathbb{E}^T)$ and $K^{[-1,0]}(\mathcal{A})$ above, the only projective-injectives are the zero objects.
 Such 0-Auslander extriangulated categories are called \defn{reduced 0-Auslander}.
 If $\mathscr{C}$ is any 0-Auslander extriangulated category then the ideal quotient of $\mathscr{C}$ by the morphisms factoring through a projective-injective object is a reduced 0-Auslander extriangulated category.
\end{remark}

\begin{figure}[h]
\begin{center}
 \capstart
  \includegraphics[scale=.8]{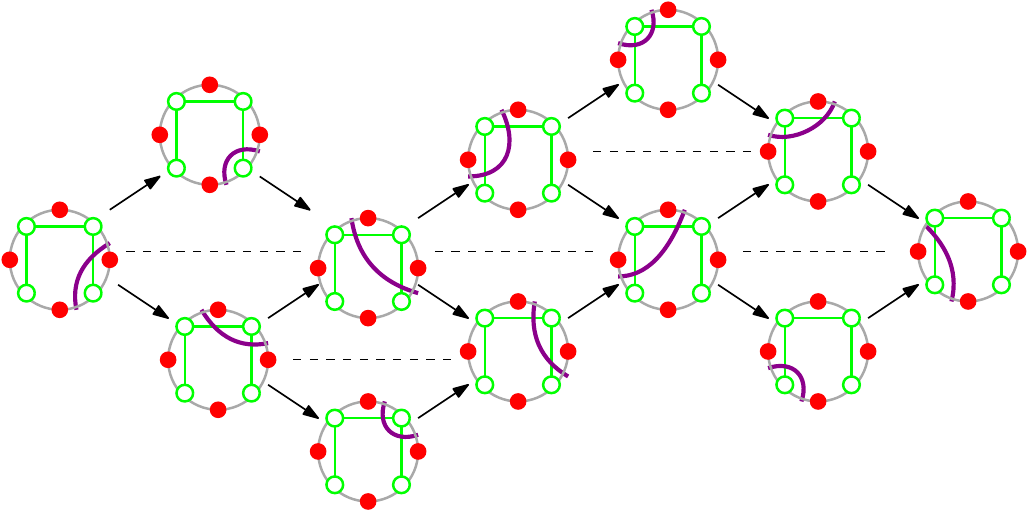}
  \caption{The Auslander--Reiten quiver of the category $\mathscr{W}$ associated with the gentle bound quiver $(Q,I)=(1\xrightarrow{\alpha}2\xrightarrow{\beta}3, \alpha\beta)$. Its vertices are in bijection with the accordions (curvy and purple) with respect to the base dissection (straight and green) corresponding to $(Q,I)$.}
  \label{fig:ARquiverAccordions}
\end{center}
\end{figure}

\subsection{Mutation of maximal rigid objects}

Fix a field $\mathbb{K}$ and a $\mathbb{K}$-linear, Krull--Schmidt, Hom-finite, 0-Auslander extriangulated category $\mathscr{C}$.

Recall that an object $R\in\mathscr{C}$ is rigid if $\mathbb{E}(R,R)=0$.

\begin{definition}
\label{def:tilting}
Let $R\in\mathscr{C}$ be rigid. We say that $R$ is
\begin{enumerate}[(i)]
 \item \defn{maximal rigid} if $R\oplus X$ is rigid only if $X\in\operatorname{add}R$;
 \item \defn{$\mathbb{E}$-tilting} if the conditions $\mathbb{E}(R,X)=0=\mathbb{E}(X,R)$ imply $X\in\operatorname{add}R$;
 \item \defn{tilting} if for any projective $P\in\mathscr{C}$, there is a conflation $P\infl R_0\defl R_1$ with $R_0,R_1\in\operatorname{add}R$;
 \item \defn{silting} if $\operatorname{thick}R = \mathscr{C}$, where $\operatorname{thick}R$ is the smallest thick subcategory of $\mathscr{C}$ containing $R$;
 \item \defn{complete rigid} if $\mathscr{C}$ has a projective generator $P$ and $|R|=|P|$, where $|X|$ denote the number of isomorphism classes of indecomposable summands of $X$. 
\end{enumerate}
\end{definition}

The following results are proven in~\cite[Section 4]{GorskyNakaokaPalu-Mutation}.

\begin{proposition}
 Let $R\in\mathscr{C}$ be rigid. We have:
 \begin{enumerate}
  \item Conditions \emph{(i)} to \emph{(iv)} in~\cref{def:tilting} are equivalent. 
  \item If $\mathscr{C}$ has a projective generator, they are also equivalent to condition \emph{(v)}.
  \item $R$ can be completed to a silting object via some version of Bongartz completion.
 \end{enumerate}
\end{proposition}

\begin{theorem}
\label{thm:mutation}
 Let $R$ be a basic silting object in $\mathscr{C}$ and let $X$ be an indecomposable summand of $R$ which is not projective-injective.
 Write $R=\overline{R}\oplus X$.
 Then there is a unique up to isomorphism indecomposable $Y\in\mathscr{C}$ such that $\overline{R}\oplus Y$ is silting.
 Moreover there is an exchange conflation $X\infl E\defl Y$ or $Y\infl E'\defl X$ with $E,E'\in\operatorname{add}\overline{R}$, but not both.
\end{theorem}

\begin{remark}
 This theorem is proven with less assumptions on $\mathscr{C}$ and allows $R$ to be a full additive subcategory.
 In the independant article~\cite{AdachiTsukamoto-silting}, Takahide Adachi and Mayu Tsukamoto study mutations of silting objects in more general extriangulated categories.
 \cref{thm:mutation} can be though of as a two-term version of their result.
 The next subsection shows that our setting is presumably the right level of generality for studying mutations in representation theory.
\end{remark}

\begin{example}
\label{ex:max rigid in 0-Auslander}
 \cref{thm:mutation} applies to the examples listed in \cref{ex:0-Auslander}
\begin{enumerate}
 \item[($\alpha$)] For $\mathscr{C}=\operatorname{rep}_\mathbb{K}\overrightarrow{A_n}$, the silting objects are the tilting representations.
 An almost complete tilting representation has one or two complements to a tilting representation.
 It has two complements precisely when it is sincere, which holds if and only if it contains the projective-injective.
 Those classical results are consequences of \cref{thm:mutation}.
 \item[($\beta$)] For $\mathscr{C}=\mathscr{E}$ as in (ii) of \cref{ex:0-Auslander}, the silting objects are the maximal almost-rigid modules of~\cite{BarnardGunawanMeehanSchiffler}.
 This suggests that using 0-Auslander exact categories might help proving mutation of maximal almost-rigid modules over other classes of algebras.
 \item[($\gamma$)] For $\mathscr{C}$ a cluster category (or a 2-Calabi--Yau triangulated category) with the extriangulated structure $\mathbb{E}^T$ given by some cluster tilting object $T$, the silting objects are the cluster tilting objects and \cref{thm:mutation} recovers classical results by~\cite{BuanMarshReinekeReitenTodorov,IyamaYoshino}. For non 2-Calabi--Yau triangulated (or extriangulated) categories, a similar approach recovers mutation of relative tilting objects.
 \item[($\delta$)] For $\mathscr{C}=K^{[-1,0]}(\operatorname{proj}\Lambda)$ the silting objects are the two-term silting objects of $K^b(\operatorname{proj}\Lambda)$ whose mutations were studied by~\cite{AiharaIyama,Kimura}.
 The importance of two-term silting complexes stems from the bijection, induced by $H^0$, with support $\tau$-tilting modules~\cite{AdachiIyamaReiten}.
 \item[($\varepsilon$)] For $\mathscr{C}=\mathscr{W}$, the silting objects are the maximal collection of pairwise non-crossing accordions or, equivalently, the non-kissing facets.
 Their mutation thus categorifies the flip of non-kissing facets~\cite{McConville,PaluPilaudPlamondon-nonkissing} and the quotient of $\mathscr{W}$ by its projective-injectives categorifies the reduced non-kissing complex.
 \item[($\zeta$)] Let $\mathscr{C}$ be the extended coheart of a co-$t$-structure $(\mathcal{A},\mathcal{B})$.
 Combining~\cite[Theorem~5.5]{AdachiTsukamoto} and~\cite[Theorem 2.1]{PauksztelloZvonareva} we obtain a bijection between the silting objects of $\mathscr{C}$ and those co-$t$-structures $(\mathcal{A}',\mathcal{B}')$ which are \defn{intermediate} with respect to $(\mathcal{A},\mathcal{B})$, i.e. such that $\mathcal{A}\subseteq\mathcal{A}'\subseteq\Sigma\mathcal{A}$.
 \cref{thm:mutation} thus extends from~\cite{BrustleYang} the applicability of mutation for intermediate co-$t$-structures.
\end{enumerate}
\end{example}

\section*{Conclusion}

Since most, if not all, extriangulated categories arising in representation theory are topological (or even algebraic), they are extension-closed, full subcategories of triangulated categories.
One might thus argue that the axiomatics of extriangulated categories is dispensable.
However extriangulated structures are very versatile as illustrated by the examples encountered in representation theory.
The following are all examples of extriangulated categories naturally arising:
\begin{enumerate}[(a)]
 \item exact categories;
 \item triangulated categories;
 \item extension-closed subcategories of triangulated categories;
 \item quotients of exact categories by some projective-injective objects;
 \item triangulated categories endowed with some relative extriangulated structure;
 \item quotients of exact categories by morphisms with injective domains and projective codomains.
\end{enumerate}

Note that, for points (d), (e), (f), no ambiant triangulated category in which the extriangulated category is ext-closed is given a priori.
Even when it is possible to construct one, this ambiant triangulated category might not enjoy some of the desirable properties (being Hom-finite, Krull--Schmidt, weakly idempotent complete, etc...) satisfied by the categories we started with.

As a final comment, \cref{fig:conclusion} illustrates the place of extriangulated categories among other relevant categorical structures appearing in representation theory and beyond.

\begin{figure}[h]
\begin{center}
 \capstart
  \includegraphics[scale=.6]{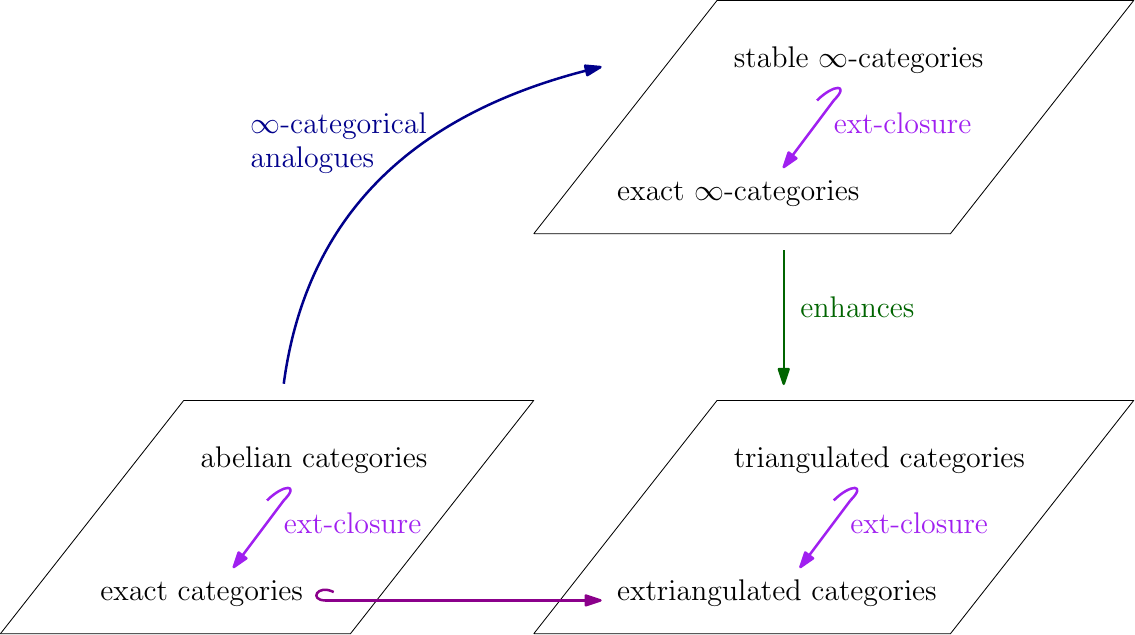}
  \caption{}
  \label{fig:conclusion}
\end{center}
\end{figure}

\bibliographystyle{alpha}
\bibliography{Alesund2022}

\newcommand{\etalchar}[1]{$^{#1}$}
\begin{thebibliography}{BMCLD{\etalchar{+}}18}

\bibitem[ABJ{\etalchar{+}}19]{ABJJLM}
P.~B. Aneesh, Pinaki Benerjee, Mrunmay Jagadale, Renjan~Rajan John, Alok
  Laddha, and Sujoy Mahato.
\newblock On positive geometries of quartic interactions {I}{I} : {S}tokes
  polytopes, lower forms on associahedra and worldsheet forms.
\newblock Preprint,
  \href{http://arxiv.org/abs/1911.06008}{\texttt{arXiv:1911.06008}}, 2019.

\bibitem[AHBHY18]{Arkani-HamedBaiHeYan}
Nima Arkani-Hamed, Yuntao Bai, Song He, and Gongwang Yan.
\newblock Scattering forms and the positive geometry of kinematics, color and
  the worldsheet.
\newblock {\em J. High Energy Phys.}, (5):096, front matter+75, 2018.

\bibitem[AHHL21]{Arkani-HamedHeLam}
Nima Arkani-Hamed, Song He, and Thomas Lam.
\newblock Cluster configuration spaces of finite type.
\newblock {\em SIGMA Symmetry Integrability Geom. Methods Appl.}, 17:Paper No.
  092, 41, 2021.

\bibitem[AI12]{AiharaIyama}
Takuma Aihara and Osamu Iyama.
\newblock Silting mutation in triangulated categories.
\newblock {\em Journal of the London Mathematical Society}, 85(3):633--668,
  2012.

\bibitem[AIR14]{AdachiIyamaReiten}
Takahide Adachi, Osamu Iyama, and Idun Reiten.
\newblock {$\tau$}-tilting theory.
\newblock {\em Compos. Math.}, 150(3):415--452, 2014.

\bibitem[Ami09]{Amiot-ClusterCats}
Claire Amiot.
\newblock Cluster categories for algebras of global dimension 2 and quivers
  with potential.
\newblock {\em Ann. Inst. Fourier (Grenoble)}, 59(6):2525--2590, 2009.

\bibitem[AT22]{AdachiTsukamoto}
Takahide Adachi and Mayu Tsukamoto.
\newblock Hereditary cotorsion pairs and silting subcategories in
  extriangulated categories.
\newblock {\em J. Algebra}, 594:109--137, 2022.

\bibitem[AT23]{AdachiTsukamoto-silting}
Takahide Adachi and Mayu Tsukamoto.
\newblock An assortment of properties of silting subcategories of
  extriangulated categories.
\newblock Preprint,
  \href{https://arxiv.org/abs/2303.08125}{\texttt{arXiv:2303.08125}}, 2023.

\bibitem[Bar15]{Barwick-exact}
Clark Barwick.
\newblock On exact $\infty$-categories and the theorem of the heart.
\newblock {\em Compositio Mathematica}, 151(11):2160--2186, 2015.

\bibitem[BCS21]{BaurCoelhoSimoes}
Karin Baur and Raquel Coelho~Sim{\~{o}}es.
\newblock A geometric model for the module category of a gentle algebra.
\newblock {\em Int. Math. Res. Not. IMRN}, (15):11357--11392, 2021.

\bibitem[BGH14]{BravoGillespieHovey-StableModule}
Daniel Bravo, James Gillespie, and Mark Hovey.
\newblock The stable module category of a general ring.
\newblock Preprint,
  \href{https://arxiv.org/abs/1405.5768}{\texttt{arXiv:1405.5768}}, 2014.

\bibitem[BGMS23]{BarnardGunawanMeehanSchiffler}
Emily Barnard, Emily Gunawan, Emily Meehan, and Ralf Schiffler.
\newblock Cambrian combinatorics on quiver representations (type {$\Bbb{A}$}).
\newblock {\em Adv. in Appl. Math.}, 143:Paper No. 102428, 37, 2023.

\bibitem[BHRS23]{BrustleHansonRoySchiffler-AlmostRigid}
Thomas Br\"ustle, Eric Hanson, Sunny Roy, and Ralf Schiffler.
\newblock Private communication in {O}berwolfach.
\newblock 2023.

\bibitem[BM21]{BazierMatte-thesis}
V\'eronique Bazier-Matte.
\newblock {\em Combinatoire des alg\`ebres amass\'ees}.
\newblock PhD thesis, Universit\'e du Qu\'ebec \`A Montr\'eal, 2021.

\bibitem[BMCLD{\etalchar{+}}18]{BazierMatteChapelierLagetDouvilleMousavandThomasYildirim}
V\'eronique Bazier-Matte, Nathan Chapelier-Laget, Guillaume Douville, Kaveh
  Mousavand, Hugh Thomas, and Emine Y\i{}ld\i{}r\i{}m.
\newblock {ABHY} {A}ssociahedra and {N}ewton polytopes of ${F}$-polynomials for
  finite type cluster algebras.
\newblock Preprint,
  \href{http://arxiv.org/abs/1808.09986}{\texttt{arXiv:1808.09986}}, 2018.

\bibitem[BMR{\etalchar{+}}06]{BuanMarshReinekeReitenTodorov}
Aslak~Bakke Buan, Robert Marsh, Markus Reineke, Idun Reiten, and Gordana
  Todorov.
\newblock Tilting theory and cluster combinatorics.
\newblock {\em Adv. Math.}, 204(2):572--618, 2006.

\bibitem[Bon10]{Bondarko}
M.~V. Bondarko.
\newblock Weight structures vs. {$t$}-structures; weight filtrations, spectral
  sequences, and complexes (for motives and in general).
\newblock {\em J. K-Theory}, 6(3):387--504, 2010.

\bibitem[BR07]{BeligiannisReiten}
Apostolos Beligiannis and Idun Reiten.
\newblock Homological and homotopical aspects of torsion theories.
\newblock {\em Mem. Amer. Math. Soc.}, 188(883):viii+207, 2007.

\bibitem[BSZ05]{BautistaSalorioZuazua}
Raymundo Bautista, Maria Jose~Souto Salorio, and Rita Zuazua.
\newblock Almost split sequences of complexes of fixed size.
\newblock {\em J. Algebra}, 287(1):140--168, 2005.

\bibitem[BT21]{BorveTrygsland}
Erlend~D. B{\o}rve and Paul Trygsland.
\newblock A theorem of {R}etakh for exact $\infty$-categories and higher
  extension functors.
\newblock Preprint,
  \href{https://arxiv.org/abs/2110.05138}{\texttt{arXiv:2110.05138}}, 2021.

\bibitem[BTHSS22]{Bennett-TennenhausHauglandSandoyShah}
Raphael Bennett-Tennenhaus, Johanne Haugland, Mads~Hustad Sand{\o}y, and Amit
  Shah.
\newblock The category of extensions and a characterisation of $n$-exangulated
  functors.
\newblock Preprint,
  \href{https://arxiv.org/abs/2205.03097}{\texttt{arXiv:2205.03097}}, 2022.

\bibitem[BTS21]{Bennett-TennenhausShah-transport}
Raphael Bennett-Tennenhaus and Amit Shah.
\newblock Transport of structure in higher homological algebra.
\newblock {\em J. Algebra}, 574:514--549, 2021.

\bibitem[B{\"{u}}h10]{Buhler-Exact}
Theo B{\"{u}}hler.
\newblock Exact categories.
\newblock {\em Expo. Math.}, 28(1):1--69, 2010.

\bibitem[BY13]{BrustleYang}
Thomas Br\"{u}stle and Dong Yang.
\newblock Ordered exchange graphs.
\newblock In {\em Advances in representation theory of algebras}, EMS Ser.
  Congr. Rep., pages 135--193. Eur. Math. Soc., Z\"{u}rich, 2013.

\bibitem[Cag18]{Cagne-PhD}
Pierre Cagne.
\newblock {\em Towards a homotopical algebra of dependent types}.
\newblock PhD thesis, Universit\'e Paris Diderot, 2018.

\bibitem[CFZ02]{ChapotonFominZelevinsky}
Fr{\'e}d{\'e}ric Chapoton, Sergey Fomin, and Andrei Zelevinsky.
\newblock Polytopal realizations of generalized associahedra.
\newblock {\em Canad. Math. Bull.}, 45(4):537--566, 2002.

\bibitem[Che23]{Chen-PhD}
Xiaofa Chen.
\newblock On exact dg categories.
\newblock PhD thesis,
  \href{https://arxiv.org/abs/2306.08231}{\texttt{arXiv:2306.08231}}, 2023.

\bibitem[Chh19]{Chhatoi-Halohedron}
Saroj Chhatoi.
\newblock A note on convex realization of halohedron.
\newblock Preprint,
  \href{http://arxiv.org/abs/1910.13786}{\texttt{arXiv:1910.13786}}, 2019.

\bibitem[DK19]{DyckerhoffKapranov}
Tobias Dyckerhoff and Mikhail Kapranov.
\newblock {\em Higher {S}egal spaces}, volume 2244 of {\em Lecture Notes in
  Mathematics}.
\newblock Springer, Cham, 2019.

\bibitem[Egg06]{Egger}
J.M. Egger.
\newblock Quillen model categories without equalisers or coequalisers.
\newblock Preprint,
  \href{https://arxiv.org/abs/math/0609808}{\texttt{arXiv:0609808}}, 2006.

\bibitem[FK10]{FuKeller}
Changjian Fu and Bernhard Keller.
\newblock On cluster algebras with coefficients and 2-{C}alabi-{Y}au
  categories.
\newblock {\em Transactions of the American Mathematical Society},
  362(2):859--895, 2010.

\bibitem[Gil11]{Gillespie-exact}
James Gillespie.
\newblock Model structures on exact categories.
\newblock {\em J. Pure Appl. Algebra}, 215(12):2892--2902, 2011.

\bibitem[Gil15]{Gillespie-HoveyTriple}
James Gillespie.
\newblock How to construct a {H}ovey triple from two cotorsion pairs.
\newblock {\em Fund. Math.}, 230(3):281--289, 2015.

\bibitem[Gil16]{Gillespie-Hereditary}
James Gillespie.
\newblock Hereditary abelian model categories.
\newblock {\em Bull. Lond. Math. Soc.}, 48(6):895--922, 2016.

\bibitem[GNP23]{GorskyNakaokaPalu-Mutation}
Mikhail Gorsky, Hiroyuki Nakaoka, and Yann Palu.
\newblock Hereditary extriangulated categories: silting objects, mutation,
  negative extensions.
\newblock Preprint,
  \href{https://arxiv.org/abs/2303.07134}{\texttt{arXiv:2303.07134}}, 2023.

\bibitem[Hap88]{Happel-book}
Dieter Happel.
\newblock {\em Triangulated categories in the representation theory of
  finite-dimensional algebras}, volume 119 of {\em London Mathematical Society
  Lecture Note Series}.
\newblock Cambridge University Press, Cambridge, 1988.

\bibitem[Hau21]{Haugland-n-exangulated}
Johanne Haugland.
\newblock The {G}rothendieck group of an {$n$}-exangulated category.
\newblock {\em Appl. Categ. Structures}, 29(3):431--446, 2021.

\bibitem[HHZ22]{HeHeZhou-localization}
Jian He, Jing He, and Panyue Zhou.
\newblock Localization of $n$-exangulated categories.
\newblock Preprint,
  \href{https://arxiv.org/abs/2205.07644}{\texttt{arXiv:2205.07644}}, 2022.

\bibitem[HJr19]{HolmJorgensen-Model}
Henrik Holm and Peter J\o~rgensen.
\newblock Model categories of quiver representations.
\newblock {\em Adv. Math.}, 357:106826, 46, 2019.

\bibitem[HLN21]{HerschendLiuNakaoka-I}
Martin Herschend, Yu~Liu, and Hiroyuki Nakaoka.
\newblock {$n$}-exangulated categories ({I}): {D}efinitions and fundamental
  properties.
\newblock {\em J. Algebra}, 570:531--586, 2021.

\bibitem[Hov02]{Hovey-CPmodelRT}
Mark Hovey.
\newblock Cotorsion pairs, model category structures, and representation
  theory.
\newblock {\em Math. Z.}, 241(3):553--592, 2002.

\bibitem[Hov07]{Hovey-CPmodel}
Mark Hovey.
\newblock Cotorsion pairs and model categories.
\newblock In {\em Interactions between homotopy theory and algebra}, volume 436
  of {\em Contemp. Math.}, pages 277--296. Amer. Math. Soc., Providence, RI,
  2007.

\bibitem[HvR19]{HenrardVanRoosmalen-localisation}
Ruben Henrard and Adam-{Christian} van Roosmalen.
\newblock Localizations of (one-sided) exact categories.
\newblock Preprint,
  \href{https://arxiv.org/abs/1903.10861}{\texttt{arXiv:1903.10861}}, 2019.

\bibitem[INP18]{IyamaNakaokaPalu}
Osamu Iyama, Hiroyuki Nakaoka, and Yann Palu.
\newblock Auslander--{R}eiten theory in extriangulated categories.
\newblock Preprint,
  \href{https://arxiv.org/abs/1805.03776}{\texttt{arXiv:1805.03776}}, to appear
  in \emph{Trans. Amer. Math. Soc.}, 2018.

\bibitem[IY08]{IyamaYoshino}
Osamu Iyama and Yuji Yoshino.
\newblock Mutation in triangulated categories and rigid {C}ohen--{M}acaulay
  modules.
\newblock {\em Inventiones mathematicae}, 172(1):117--168, 2008.

\bibitem[IY20]{IyamaYang}
Osamu Iyama and Dong Yang.
\newblock Quotients of triangulated categories and equivalences of {B}uchweitz,
  {O}rlov, and {A}miot-{G}uo-{K}eller.
\newblock {\em Amer. J. Math.}, 142(5):1641--1659, 2020.

\bibitem[Jin20]{Jin}
Haibo Jin.
\newblock Cohen-{M}acaulay differential graded modules and negative
  {C}alabi-{Y}au configurations.
\newblock {\em Adv. Math.}, 374:107338, 59, 2020.

\bibitem[JKPW23]{JassoKvammePaluWalde}
Gustavo Jasso, Sondre Kvamme, Yann Palu, and Tashi Walde.
\newblock Homotopical algebra in exact $\infty$-categories.
\newblock Work in progress, 2023.

\bibitem[JL22]{JagadaleLaddha}
Mrunmay Jagadale and Alok Laddha.
\newblock Towards positive geometry of multi scalar field amplitudes.
  {A}ccordiohedron and effective field theory.
\newblock {\em J. High Energy Phys.}, (4):Paper No. 100, 37, 2022.

\bibitem[JM17]{JacquetMalo-model}
Lucie Jacquet-Malo.
\newblock Model categories structures from rigid objects in exact categories.
\newblock Preprint,
  \href{http://arxiv.org/abs/1706.06530}{\texttt{arXiv:1706.06530}}, 2017.

\bibitem[Kel05]{Keller-orbit}
Bernhard Keller.
\newblock On triangulated orbit categories.
\newblock {\em Documenta Mathematica}, 10:551--581, 2005.

\bibitem[Kim20]{Kimura}
Yuta Kimura.
\newblock Tilting theory of noetherian algebras.
\newblock Preprint, \href{https://arxiv.org/abs/2006.01677}
  {\texttt{arXiv:2006.01677}}, 2020.

\bibitem[Kla22]{Klapproth-nExtClosed}
Carlo Klapproth.
\newblock $n$-extension closed subcategories of $n$-exangulated categories.
\newblock Preprint,
  \href{http://arxiv.org/abs/2209.01128}{\texttt{arXiv:2209.01128}}, 2022.

\bibitem[Kle22]{Klemenc}
Jona Klemenc.
\newblock The stable hull of an exact $\infty$-category.
\newblock {\em Homology, Homotopy and Applications}, 24(2):195--220, 2022.

\bibitem[McC17]{McConville}
Thomas McConville.
\newblock Lattice structure of {G}rid-{T}amari orders.
\newblock {\em J. Combin. Theory Ser. A}, 148:27--56, 2017.

\bibitem[McM73]{McMullen-typeCone}
Peter McMullen.
\newblock Representations of polytopes and polyhedral sets.
\newblock {\em Geometriae Dedicata}, 2:83--99, 1973.

\bibitem[MP12]{MayPonto-MoreConcise}
J.~Peter May and Kathleen Ponto.
\newblock {\em More concise algebraic topology}.
\newblock Chicago Lectures in Mathematics. University of Chicago Press,
  Chicago, IL, 2012.
\newblock Localization, completion, and model categories.

\bibitem[Msa22]{Msapato}
Dixy Msapato.
\newblock The {K}aroubi envelope and weak idempotent completion of an
  extriangulated category.
\newblock {\em Appl. Categ. Structures}, 30(3):499--535, 2022.

\bibitem[Nak18]{Nakaoka-simultaneous}
Hiroyuki Nakaoka.
\newblock A simultaneous generalization of mutation and recollement of
  cotorsion pairs on a triangulated category.
\newblock {\em Appl. Categ. Structures}, 26(3):491--544, 2018.

\bibitem[NOS22]{NakaokaOgawaSakai}
Hiroyuki Nakaoka, Yasuaki Ogawa, and Arashi Sakai.
\newblock Localization of extriangulated categories.
\newblock {\em Journal of Algebra}, 611:341--398, 2022.

\bibitem[NP19]{NakaokaPalu}
Hiroyuki Nakaoka and Yann Palu.
\newblock Extriangulated categories, {H}ovey twin cotorsion pairs and model
  structures.
\newblock {\em Cah. Topol. G\'{e}om. Diff\'{e}r. Cat\'{e}g.}, {L}{X}, 2019.

\bibitem[NP20]{NakaokaPalu2}
Hiroyuki Nakaoka and Yann Palu.
\newblock External triangulation of the homotopy category of an exact
  quasi-category.
\newblock Preprint,
  \href{https://arxiv.org/abs/2004.02479}{\texttt{arXiv:2004.02479}}, 2020.

\bibitem[OPS18]{OpperPlamondonSchroll}
Sebastian Opper, Pierre-Guy Plamondon, and Sibylle Schroll.
\newblock A geometric model for the derived category of gentle algebras.
\newblock Preprint,
  \href{http://arxiv.org/abs/1801.09659}{\texttt{arXiv:1801.09659}}, 2018.

\bibitem[Pal08]{Palu-CC}
Yann Palu.
\newblock Cluster characters for 2-{C}alabi-{Y}au triangulated categories.
\newblock {\em Ann. Inst. Fourier (Grenoble)}, 58(6):2221--2248, 2008.

\bibitem[Pau08]{Pauksztello}
David Pauksztello.
\newblock Compact corigid objects in triangulated categories and
  co-{$t$}-structures.
\newblock {\em Cent. Eur. J. Math.}, 6(1):25--42, 2008.

\bibitem[PPP19]{PaluPilaudPlamondon-surfaces}
Yann Palu, Vincent Pilaud, and Pierre-Guy Plamondon.
\newblock Non-kissing and non-crossing complexes for locally gentle algebras.
\newblock {\em Journal of Combinatorial Algebra}, 3(4):401--438, 2019.

\bibitem[PPP21]{PaluPilaudPlamondon-nonkissing}
Yann Palu, Vincent Pilaud, and Pierre-Guy Plamondon.
\newblock Non-kissing complexes and tau-tilting for gentle algebras.
\newblock {\em Mem. Amer. Math. Soc.}, 274(1343):vii+110, 2021.

\bibitem[PPPP19]{PadrolPaluPilaudPlamondon}
Arnau Padrol, Yann Palu, Vincent Pilaud, and Pierre-Guy Plamondon.
\newblock Associahedra for finite type cluster algebras and minimal relations
  between g-vectors.
\newblock Preprint,
  \href{http://arxiv.org/abs/1906.06861}{\texttt{arXiv:1906.06861}}, to appear
  in \emph{Proc. London Maht. Soc.}, 2019.

\bibitem[Pre17]{Pressland}
Matthew Pressland.
\newblock A categorification of acyclic principal coefficient cluster algebras.
\newblock Preprint,
  \href{https://arxiv.org/abs/1702.05352}{\texttt{arXiv:1702.05352}}, 2017.

\bibitem[PZ20]{PauksztelloZvonareva}
David Pauksztello and Alexandra Zvonareva.
\newblock Co-{$t$}-structures, cotilting and cotorsion pairs.
\newblock Preprint,
  \href{https://arxiv.org/abs/2007.06536}{\texttt{arXiv:2007.06536}}, 2020.

\bibitem[Rum21]{Rump}
Wolfgang Rump.
\newblock The acyclic closure of an exact category and its triangulation.
\newblock {\em J. Algebra}, 565:402--440, 2021.

\bibitem[{\v{S}}{\v{t}}o13]{Stovicek-ExactModel}
Jan {\v{S}}{\v{t}}ov\'{\i}\v{c}ek.
\newblock Exact model categories, approximation theory, and cohomology of
  quasi-coherent sheaves.
\newblock In {\em Advances in representation theory of algebras}, EMS Ser.
  Congr. Rep., pages 297--367. Eur. Math. Soc., Z\"{u}rich, 2013.

\bibitem[Tac64]{Tachikawa}
Hiroyuki Tachikawa.
\newblock On dominant dimensions of {Q}{F}-3 algebras.
\newblock {\em Transactions of the American Mathematical Society},
  112(2):249--266, 1964.

\bibitem[Tat22]{Tattar-thesis}
Aran Tattar.
\newblock Torsion structures, subobjects and unique filtrations in non-abelian
  categories.
\newblock \href{https://kups.ub.uni-koeln.de/61013/}{\texttt{PhD thesis}},
  {U}niversit{\"a}t zu {K}{\"o}ln, 2022.

\bibitem[Wu21]{Wu}
Yilin Wu.
\newblock Relative cluster categories and higgs categories.
\newblock Preprint,
  \href{https://arxiv.org/abs/2109.03707}{\texttt{arXiv:2109.03707}}, 2021.

\bibitem[Yan15]{Yang-TriangulatedHovey}
Xiaoyan Yang.
\newblock Model structures on triangulated categories.
\newblock {\em Glasg. Math. J.}, 57(2):263--284, 2015.

\end{thebibliography}

\end{document}